\documentclass[a4paper,12pt]{amsart}
\usepackage[T1]{fontenc}
\usepackage[utf8]{inputenc}
\usepackage[margin=2.5cm]{geometry}

\usepackage{amsmath, amssymb,amsfonts}
\usepackage{thm-restate}
\usepackage{array}
\usepackage{multirow}
\usepackage[dvipsnames]{xcolor}
\usepackage{tikz}
\usetikzlibrary{shapes.misc}
\usetikzlibrary{shapes.geometric}
\usetikzlibrary {decorations.pathmorphing}

\tikzset{unseulcoin fill/.style={append after command={
   \pgfextra
        \draw[sharp corners, fill=#1, color=#1, line width = 0mm]%
    (\tikzlastnode.west)%
    [rounded corners=0pt] |- (\tikzlastnode.north)%
    [rounded corners=0pt] -| (\tikzlastnode.east)%
    [rounded corners=0pt] |- (\tikzlastnode.south)%
    [rounded corners=5pt] -| (\tikzlastnode.west);
   \endpgfextra}}}

\tikzset{black node/.style = {draw=black,fill=black,circle, inner sep=0, minimum size=0.15cm}}
\usepackage{multicol}
\usepackage{paralist}
\usepackage{hyperref}
\hypersetup{
  pdftitle = {Long induced paths in sparse graphs and graphs with forbidden patterns},
  pdfauthor = {J. Duron, L. Esperet, J.-F. Raymond},
  colorlinks = true,
  linkcolor = black!30!red,
  citecolor = black!30!green
}
\usepackage{cleveref}
\usepackage{dirtytalk}

\newtheorem{theorem}{Theorem}[section]
\newtheorem{corollary}[theorem]{Corollary}

\newtheorem{lemma}[theorem]{Lemma}

\newtheorem{observation}[theorem]{Observation}
\newtheorem{remark}[theorem]{Remark}

\def\cqedsymbol{\ifmmode$\lrcorner$\else{\unskip\nobreak\hfil
\penalty50\hskip1em\null\nobreak\hfil$\lrcorner$
\parfillskip=0pt\finalhyphendemerits=0\endgraf}\fi}

\DeclareMathOperator{\col}{col}

\DeclareMathOperator{\ter}{Ter}
\newcommand{\N}{\mathbb{N}}

\newcommand{\intv}[2]{\left \{ #1,\dots, #2\right \}}
\newcommand{\Pleft}{P^{\rm left}}
\newcommand{\Pright}{P^{\rm right}}
\newcommand{\Pmid}{P^{\rm mid}}
\newcommand{\rev}{{\rm rev}}

\newcommand{\C}{\mathcal{C}}

\DeclareMathOperator{\polyloglog}{polyloglog}

\DeclareMathOperator{\tw}{\bf tw}
\DeclareMathOperator{\pw}{\bf pw}

\newcommand{\pow}{*}

\let\le\leqslant
\let\ge\geqslant
\let\leq\leqslant
\let\geq\geqslant

\title[Long induced paths in sparse graphs and graphs with forbidden patterns]{Long induced paths in sparse graphs\\ and graphs with forbidden patterns}

\author[J.~Duron]{Julien Duron}
\address[J.~Duron]{Univ.\ Lyon, CNRS, ENS de Lyon, Université Claude Bernard Lyon 1, LIP UMR5668,
  Lyon, France}
\email{j.duron@uw.edu.pl}

\author[L.~Esperet]{Louis Esperet}
\address[L.~Esperet]{Univ.\ Grenoble Alpes, CNRS, Laboratoire G-SCOP,
  Grenoble, France}
\email{louis.esperet@grenoble-inp.fr}

\author[J.-F.~Raymond]{Jean-Florent Raymond}
\address[J.-F.~Raymond]{Univ.\ Lyon, CNRS, ENS de Lyon, Université Claude Bernard Lyon 1, LIP UMR5668,
  Lyon, France}
\email{jean-florent.raymond@cnrs.fr}

\date{\today}

\thanks{The authors are partially supported by the French ANR Projects TWIN-WIDTH
  (ANR-21-CE48-0014-01) and GRALMECO (ANR-21-CE48-0004), and by LabEx
  PERSYVAL-lab (ANR-11-LABX-0025).}

\begin{document}

\maketitle
\begin{abstract}
Consider a graph $G$ with a path $P$ of order $n$. What conditions force $G$ to also have a long \emph{induced} path? As complete bipartite graphs have long paths but no long induced paths, a natural restriction is to forbid some fixed complete bipartite graph $K_{t,t}$ as a subgraph. In this case we show that $G$ has an induced path of order $(\log \log n)^{1/5-o(1)}$.  This is an exponential improvement over a result of Galvin, Rival, and Sands (1982) and comes close to a recent upper bound of order $O((\log \log n)^2)$. 

Another way to approach this problem is by viewing $G$ as an ordered graph (where the vertices are ordered according to their position on the path $P$). From this point of view it is most natural to consider which ordered subgraphs need to be forbidden in order to force the existence of a long induced path. Focusing on the exclusion of ordered matchings, we improve or recover a number of existing results with much simpler proofs, in a unified way. We also show that if some forbidden ordered subgraph forces the existence of a long induced path in $G$, then this induced path has size at least $\Omega((\log \log \log n)^{1/3})$, and can be chosen to be increasing with respect to $P$. 
\end{abstract}

\section{Introduction}

The starting point of this work is the following result proved by Galvin, Rival, and Sands in 1982.

\begin{theorem}[\cite{galvin1982ramsey}]\label{th:galvin}
For every $t\in \N$ there is an unbounded function $f\colon \N\to\N$ such that for every graph $G$ the following holds: if $G$ is $K_{t,t}$-subgraph free and has a path of order $n$, then $G$ has an induced path of order at least $f(n)$.
\end{theorem}

This result states that in the absence of large $K_{t,t}$ subgraphs, long induced paths are unavoidable induced subgraphs of graphs that contain large paths.
The function $f$ is not given explicitly in \cite{galvin1982ramsey}, but it directly follows from the proof there and known bounds on multicolor Ramsey numbers for quadruples that $f(n)=\Omega((\log \log \log n)^{1/3})$ (see Theorem \ref{thm:GRS} and its proof in Section~\ref{sec:triple}).
We may wonder whether this result is best possible. For hereditary graph classes, being $K_{t,t}$-subgraph free is actually necessary. Indeed, complete (bipartite) graphs have a path that visit all their vertices, yet they do not have induced paths on more than 3 vertices. If a hereditary graph class $\mathcal{C}$ has graphs with arbitrarily large $K_{t,t}$ subgraphs, then by Ramsey's theorem $\mathcal{C}$ contains arbitrarily large complete graphs or complete bipartite graphs so as observed above the outcome of \Cref{th:galvin} cannot hold. 

We note than an independent proof of \Cref{th:galvin} was given 30 years later in \cite{atminas2012linear} by Atminas, Lozin, and Razgon, who were not aware of the original result.\footnote{Personal communication.} They did not try to optimize $f$ and indeed a quick glance at their proof suggests a function with at least 8 nested logarithms.

The question of the optimal function $f$ for \Cref{th:galvin} can be stated more generally as follows.

\begin{restatable}{question}{fstquestion}\label{question}
Given a graph class $\mathcal{C}$, what is the maximum function $f_\mathcal{C}\colon \N\to\N$ such that the following property holds?
\begin{itemize}
    \item[$(\star)$] Every $G\in \mathcal{C}$ with a path of order $n$ has an induced path of order at least $f_{\mathcal{C}}(n)$.
\end{itemize}
\end{restatable}

A visual summary of the main existing results on this question is displayed in \Cref{fig:survey}. We note that Question \ref{question} has been mostly considered for hereditary classes (i.e., classes of graphs that are closed under taking induced subgraphs).

\definecolor{rouge}{RGB}{253,82,50}
\definecolor{orang}{RGB}{253,175,50}
\definecolor{vert}{RGB}{50,225,100}

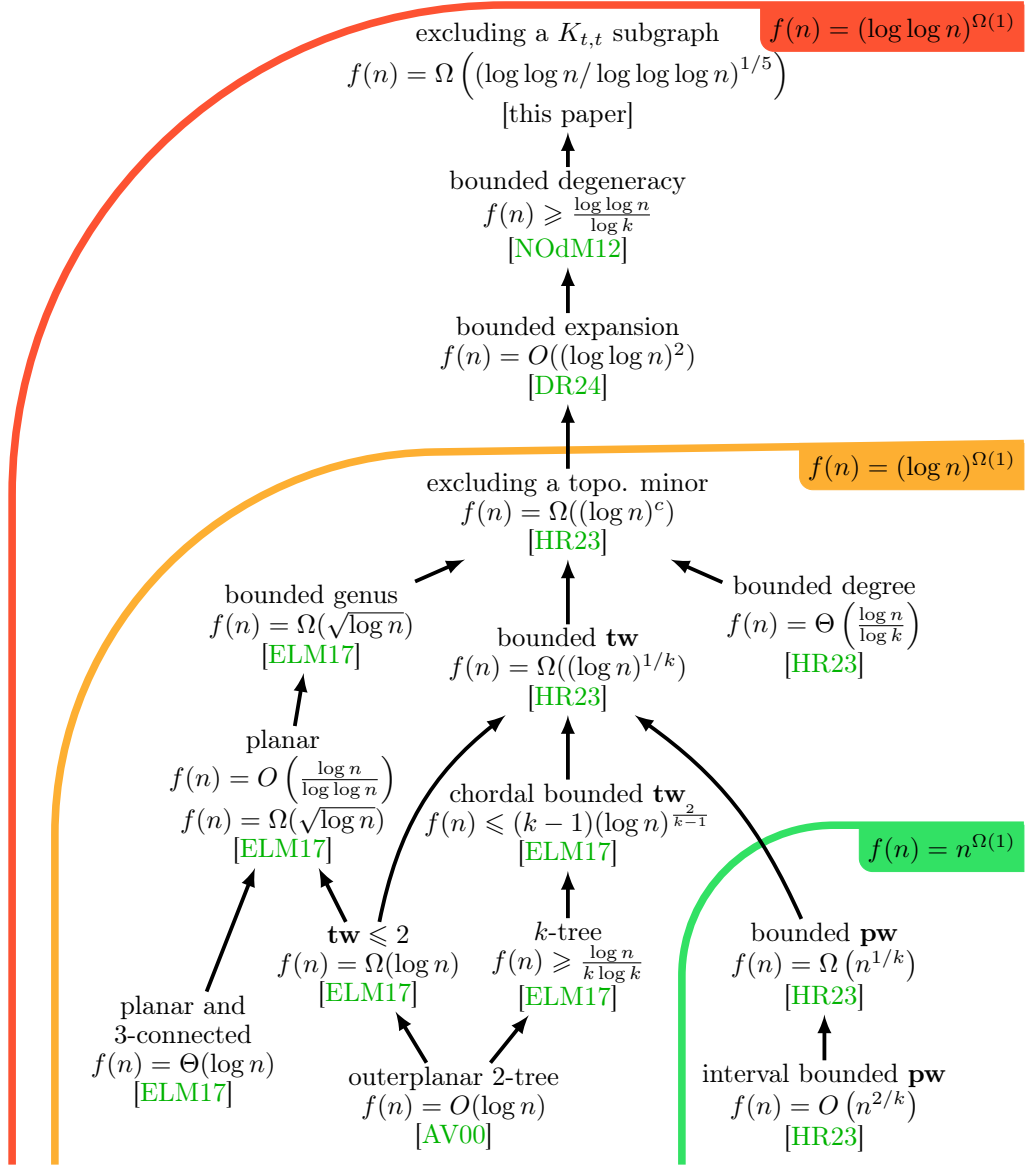
\begin{figure}[htb]
    \centering
\begin{tikzpicture}[y=-1cm, x=1cm, 
scale=0.75,
every node/.style={rectangle, inner sep=0.5mm, outer sep=0, align=center, draw=none, font={\fontsize{10}{11}\selectfont}}]

  \path[draw=rouge,line width=1mm, rounded corners=6.5cm] (0.25, 28) -- (0.25, 7.5) -- (18, 7.5);
  \path [draw=orang,line width=0.1cm, rounded corners=5cm] (1, 28) -- (1, 15.5) -- (18, 15.25);
  \path[draw=vert, line width=0.1cm, rounded corners=2cm] (12, 28) -- (12,22) -- (18, 22);
  
\begin{scope}[every node/.style={anchor = north east, inner sep=1mm,font={\fontsize{10}{11}\selectfont}}]
  \node[unseulcoin fill=rouge] at (18,7.5){$f(n) = (\log\log n)^{\Omega(1)}$};
  \node[unseulcoin fill=orang] at (18, 15.25){$f(n) = (\log n)^{\Omega(1)}$};
  \node[unseulcoin fill=vert] (poly) at (18, 22){$f(n) = n^{\Omega(1)}$};
  \end{scope}
  
  \node (bic) at (10, 8.75){excluding a $K_{t,t}$ subgraph\\$f(n) = \Omega\left ( \left ({\log \log n/\log \log \log n}\right )^{1/5} \right )$\\{[this paper]}};

  \node (deg) at (10, 11.25){bounded degeneracy\\$f(n) \geq  \frac{\log \log n}{\log k}$\\\cite{nevsetvril2012sparsity}};

  \node (exp) at (10, 13.75){bounded expansion\\$f(n) = O((\log \log n)^2)$\\\cite{defrain2024sparse}};

  \node (topo) at (10, 16.5){excluding a topo. minor\\$f(n) = \Omega( (\log n)^{c})$\\\cite{hilaire2023long}};

  \node (degree) at (14.5, 18.5){bounded degree\\$f(n) = \Theta\left (\frac{\log n}{\log k}\right )$\\
 \cite{hilaire2023long}
 };

  \node (genus) at (5.5, 18.5){bounded genus\\$f(n) = \Omega(\sqrt{\log n})$\\\cite{esperet2017long}};

  \node (planar) at (5, 21.5){planar\\$f(n) = O\left (\frac{\log n}{\log \log n}\right )$\\$f(n) = \Omega(\sqrt{\log n})$\\\cite{esperet2017long}};

  \node (outer) at (8, 27){outerplanar 2-tree\\$f(n) = O(\log n)$\\\cite{arocha2000long}};

  \node (tw) at (10, 19.25){bounded {\bf tw}\\$f(n) = \Omega((\log n)^{1/k})$\\\cite{hilaire2023long}};

  \node (pl3c) at (3.25, 26){planar and\\3-connected\\$f(n) = \Theta(\log n)$\\\cite{esperet2017long}};

  \node (tw2) at (6.5, 24.5){$\textbf{tw} \leq 2$\\$f(n) = \Omega(\log n)$\\\cite{esperet2017long}};

  \node (ktree) at (10, 24.5){$k$-tree\\$f(n) \geq \frac{\log n}{k \log k}$\\\cite{esperet2017long}};

  \node (chor) at (10, 22){chordal bounded\ $\bf{tw}$\\$f(n) \leq (k-1) (\log n)^{\frac{2}{k-1}}$\\\cite{esperet2017long}};

  \node (pw) at (14.5, 24.5){bounded {\bf pw}\\$f(n) = \Omega\left (n^{1/k} \right)$\\\cite{hilaire2023long}};
  
  \node (intv) at (14.5, 
  27){interval bounded $\bf{pw}$\\$f(n) = O\left (n^{2/k}\right )$\\\cite{hilaire2023long}};
  
  \begin{scope}[every path/.style={latex-, draw, line width=0.5mm}]
    \path (pw) -- (intv);
    \path (tw) to[bend right=20] (tw2);
    \path (tw) -- (chor);
    \path (chor) -- (ktree);
    \path (tw2) -- (outer);
    \path (planar) -- (tw2);
    \path (planar) -- (pl3c);
    \path (genus) -- (planar);
    \path (topo) -- (genus);
    \path (topo) -- (tw);
    \path (topo) -- (degree);
    \path (exp) -- (topo);
    \path (deg) -- (exp);
    \path (bic) -- (deg);
    \path (tw) to[bend left = 15] (pw);
    \path (ktree) -- (outer);
\end{scope}

\end{tikzpicture}
    \caption{Currently best known bounds regarding \Cref{question}.
    We recall that for hereditary graph classes, excluding a bipartite subgraph is the most general setting where the property of \Cref{question} holds with $f$ unbounded. In the results above, $k$ is a strict upper-bound on the considered parameter and $c$ depends on the excluded graph. Treewidth and pathwidth are respectively abbreviated {\bf tw} and {\bf pw}. Arrows point towards more general concepts.}
    \label{fig:survey}
\end{figure}

In particular, less general than $K_{t,t}$-subgraph free graphs, the case of $d$-degenerate\footnote{A graph is \emph{$d$-degenerate} if all its subgraphs have a vertex of degree at most $d$.} graphs was studied by Ne{\v{s}}et{\v{r}}il and Ossona de Mendez~\cite{nevsetvril2012sparsity}.
They proved that if $\mathcal{C}$ is the class of $d$-degenerate graphs then $f_{\mathcal{C}}(n) \geq \log \log n/\log( d+1)$.
This is close to the correct order of magnitude as Defrain and the third author recently constructed in \cite{defrain2024sparse} 2-degenerate graphs with paths of order $n$ and where induced paths have order $O((\log \log n)^2)$.
As graphs of bounded degeneracy exclude large complete (bipartite) graphs, this upper-bound also holds for the function of \Cref{th:galvin}. Hence, prior to this paper, there was an exponential gap between the best known upper- and lower-bounds for the maximum function $f$ such that \Cref{th:galvin} holds.
In this paper we prove the following exponential improvement on \Cref{th:galvin}.

\begin{theorem}\label{th:main}
For every every positive integer $t$ there is a constant $c$ such that if a graph $G$ has a path of order $n$ and no $K_{t, t}$ subgraph, then $G$ has an induced path of order at least $c (\log \log n/\log \log \log n)^{1/5}$.
\end{theorem}

Hence the gap is now only polynomial between the upper and lower bounds for the function $f_\mathcal{C}$, where $\mathcal{C}$ is the class of graphs with no $K_{t,t}$ subgraphs.
Moreover, we have the following dichotomy: 

\begin{corollary}\label{cor:dicho_f}
There exists a constant $c$ such that, for every hereditary graph class $\C$, either $f_\C(n) = \Omega((\log\log n)^c)$ or $f_\C(n) = O(1)$.
\end{corollary}

\subsection*{Forbidden patterns} In most proofs of results in the area, what matters when considering some path $P=v_1,\ldots,v_n$ in a graph $G$ is not so much the subgraph of $G$ induced by $P$, but instead the \emph{ordered} subgraph of $G$ induced by $P$, with underlying order $v_1<\cdots<v_n$. It turns out that forbidding ordered subgraphs in $G$ (rather than subgraphs) not only allows for a much more fine-grained understanding of the function $f_\C$, but also provides a unifying framework to obtain as simple corollaries most of the results that have been proved so far in the literature.

\medskip

Let us say that a graph $G$ with a path $P=v_1,\ldots,v_n$ contains some ordered subgraph $H$ as a \emph{pattern} if $H$ appears as an ordered subgraph of $G[V(P)]-E(P)$, considered with the ordering $v_1<\cdots<v_n$. This means that the ordering of the vertices of the copy of $H$ in $G$ is the same ordering as in $G$ (along the path $P$), and moreover we do not consider the edges of $P$ as part of a pattern. If $G$ does not contain a pattern $H$, we say that $G$ \emph{avoids the pattern~$H$}.
In this setting we consider the following natural counterpart of \Cref{question}.

\begin{restatable}{question}{sndquestion}\label{question-pattern}
Given an ordered graph $H$, what is the maximum function $g_{H}\colon \N\to\N$ such that the following property holds?
\begin{itemize}
    \item[$(\star\star)$] Every graph $G$ with a path $P$ of order $n$ that avoids $H$ as a pattern has an induced path of order at least $g_{H}(n)$.
\end{itemize}
\end{restatable}

Observe that we can assume  that $P$ is a Hamiltonian path in $G$ (by considering the subgraph of $G$ induced by $P$ instead of $G$).
If $H=K_2$ then $G$ is precisely an induced path on $n$ vertices, so $g_{K_2}(n)=n$. 
Our first result is that if $g_H(n)=\omega(\log n)$, then $H$ must be a matching. Hence, better-than-logarithmic bounds on the size of induced paths can only be obtained by considering very simple patterns, namely ordered matchings. It is thus natural to investigate  $g_H(n)$ when $H$ is an ordered matching, and we prove the following: \begin{itemize}
    \item either $H$ is \emph{non-crossing} (that is, it does not contain vertices $a<b<c<d$ with edges $ac,bd$) and then $g_H(n)$ is polynomial, or
    \item $H$ contains a pair of crossing edges and then $g_H(n)$ is polylogarithmic.
\end{itemize} 

It turns out that excluding an ordered matching contains a number of earlier results as particular cases, and we therefore either recover or improve a number of known results as simple corollaries of our result:

\begin{itemize}
    \item for the class $\mathcal{C}$ of graphs of pathwidth at most $p$, $f_\mathcal{C}(n)=\Omega(n^{1/p})$ \cite{hilaire2023long};
    \item for the class $\mathcal{C}$ of graphs of treewidth at most $t$, $f_\mathcal{C}(n)=(\log n)^{\Omega(1/t)}$ \cite{hilaire2023long};
    \item for the class $\mathcal{C}$ of outerplanar graphs, $f_{\mathcal{C}}(n) = \Omega(\log n)$~\cite{esperet2017long};
    \item for the class $\mathcal{C}$ of planar graphs (or more generally for any class of graphs embeddable on a fixed surface), $f_\mathcal{C}(n)=\Omega((\log n)^{1/2})$ \cite{esperet2017long};
    \item for the class $\mathcal{C}$ of $K_t$-minor-free graphs, $f_\mathcal{C}(n)=(\log n)^{\Omega(1/t^2)}$ \cite{hilaire2023long}.
\end{itemize}

Note that in the last item in the list above (for $K_t$-minor-free graphs) the exponent in the polylogarithmic bound of \cite{hilaire2023long} was an unspecified function of $t$ (depending on the Robertson-Seymour graph minor structure theorem). Actually the result of \cite{hilaire2023long} was obtained directly for the class of graphs with no $K_t$ as a topological minor (which contains the class of $K_t$-minor-free graphs), but we do not believe that we can recover this result by only forbidding order matchings. This leads us to consider the more general case where $H$ is a  star forest in a companion paper \cite{II}. In \cite{II}  we completely characterize all ordered subgraphs $H$ such that $g_H$ is polylogarithmic. A consequence of the main result of \cite{II}  is that for the class $\mathcal{C}$ of graphs which do not contain $K_t$ as a topological minor, $f_\mathcal{C}(n)=(\log n)^{\Omega(1/(t (\log t)^2))}$. The proof of this result is significantly more involved than the proofs of the previous results involving forbidden matchings. 

\medskip

We also obtain the following simple corollary of our results.

\begin{corollary}\label{cor:minclosed}
Let $\mathcal{C}$ be a proper minor-closed graph class. If $\mathcal{C}$ excludes some outerplanar graph, then $f_\mathcal{C}(n)$ is polynomial, and otherwise $f_\mathcal{C}(n)$ is polylogarithmic. 
\end{corollary}

We conclude the paper with a complete characterization of ordered graphs $H$ such that $g_H(n)=O(1)$ and prove the following dichotomy result which is very similar to Corollary \ref{cor:dicho_f}: for any ordered graph $H$, either $g_H(n) = \Omega((\log\log \log n)^c)$ for some $c>0$, or $g_H(n) = O(1)$. The proof of one of the two directions is very similar to that of Theorem \ref{th:galvin}, which was our starting point. As a consequence of our results we obtain the following dichotomies, that reveal jumps in the growth rate of~$g_H$.
\begin{corollary}\label{cor:dicho}
Let $H$ be an ordered graph.
\begin{enumerate}
    \item $g_H(n) = n^{\Omega(1)}$ if and only if $H$ is a subgraph of a non-crossing matching, otherwise $g_H(n) = O(\log n)$;
    \item\label{e:obip} $g_H(n) = (\log \log n)^{\Omega(1)}$ if $H$ is bipartite, with one partite set preceding the other in the order;
    \item $g_H(n) = (\log \log \log n)^{\Omega(1)}$ if and only if $H$ is a subgraph of the ordered half-graph, and otherwise $g_H(n) = O(1)$.
\end{enumerate}
\end{corollary}

One remark about Corollary \ref{cor:dicho} is that there is a significant difference between case (1) and cases (2)--(3). In (1) the exponent depends on the excluded pattern $H$, while in (2)--(3) the exponent is an absolute constant, independent of the pattern.

\medskip

In the companion paper \cite{II}, we will add the following item to the dichotomies above: 

\begin{theorem}[\cite{II}]\label{th:II}
Let $H$ be an ordered graph. Then $g_H(n) = (\log n)^{\Omega(1)}$ if and only if $H$ is a constellation, and otherwise $g_H(n) = O((\log \log n)^2)$.
\end{theorem}

In the statement of Theorem \ref{th:II}, a \emph{constellation} is a star forest with a specific vertex ordering (see \cite{II} for more details).

\subsection*{Organization of the paper}
In \Cref{sec:prelim} we introduce the necessary tools and definitions. \Cref{sec:proof} is devoted to the proof of \Cref{th:main}. 
We study forbidden patterns in Section \ref{sec:pattern}, first in the case of matchings in Subsection \ref{sec:matching}, where we prove in particular Corollary \ref{cor:minclosed} and then for general patterns in Subsection \ref{sec:triple}, where we prove Corollary \ref{cor:dicho}.
We conclude with open problems in \Cref{sec:open}.

\section{Preliminaries}
\label{sec:prelim}

In all the paper, to avoid any ambiguity we always consider the \emph{order} of a path $P$ (its number of vertices), denoted by $|P|$. We never refer to the length (number of edges) of a path. 

\medskip

Logarithms are in base 2. As we will often be using several levels of exponentiations, it will sometimes be more convenient to write exponentiation in-line: we will then write $a\pow b$ instead of $a^b$. We omit parentheses for the sake of readability but it should be implicit that $\pow$ is not associative and $n_1 \pow n_2\pow \cdots  \pow n_k=n_1 \pow (n_2 \pow (\cdots \pow n_k )\cdots ) $. We use $\log^{(i)}$ to denote the logarithm iterated $i$ times. That is, $\log^{(0)}$ is the identity function and for every integer $i\geq 1$ and every $x$ such that $\log^{(i-1)} x>0$, $\log^{(i)} x = \log (\log^{(i-1)} x)$.

\medskip

We deal with simple, undirected, and loopless graphs and hypergraphs. %
A hypergraph $H$ is a \emph{$k$-clique} if its hyperedges are all the $k$-element subsets of its vertex set. For a fixed coloring of the hyperedges of $H$, we say that a subhypergraph $H'$ of $H$ is \emph{monochromatic} if all its hyperedges have the same color.
We use the following bound on  multicolored Ramsey numbers given by Erd\H{o}s and Rado.

\begin{theorem}[{\cite[Theorem~1]{erdos1952combinatorial}}]\label{th:ramseyc}
For every integers $q, k, N\geq 1$, there is a number
\[
r = R_q(N;k) \leq q \pow (q^{k-1}) \pow \cdots \pow (q^2) \pow (q(N-k)+1)
\]
such that for every coloring of the hyperedges of a $k$-clique $K$ of order $r$  with $q$ colors, $K$ contains a monochromatic $k$-clique of order $N$. Note that in particular, when $k=3$, $R_q(N;3)\le q\pow q \pow (2q(N-3)+2)\le q\pow q \pow (2qN)$.
\end{theorem}

\section{\texorpdfstring{Proof of \Cref{th:main}}{Proof of Theorem 1}}
\label{sec:proof}

In this section we prove \Cref{th:main}. The reader is invited to start reading the much simpler proof of Theorem \ref{thm:GRS} (which is very similar to the original proof of Theorem \ref{th:galvin} in \cite{galvin1982ramsey}) as a comparison. 

\medskip

Observe that we only need to prove the statement on graphs with a Hamiltonian path of order $n$ as we can simply ignore the vertices not belonging to the path of order $n$.
So we consider a graph $G$ with a Hamiltonian path of order $n$.
We define $\prec$ as an ordering of $V(G)$ according to this Hamiltonian path.
A path $u_1\dots u_\ell$ of $G$ is \emph{increasing} if $u_1 \prec \dots \prec u_\ell$. In what follows, by \emph{largest vertex} we always mean largest with respect to $\prec$. An \emph{ordered triple} of vertices of $G$ is a triple $(u,v,w) \in V(G)^3$ where $u\prec v \prec w$.

For any integer $t$, let $s$ be the largest integer such that 
\[n\geq  (10^3 s^4) \pow (10^3 s^4) \pow (2 \cdot 10^3 s^4 (\max(s, 2t + 5))).\] Note that  $s=\Omega((\log^{(2)} n/\log^{(3)}n)^{1/5})$, where the hidden multiplicative constant depends on $t$.
We will show that $G$ contains a $K_{t, t}$ subgraph or an increasing induced path of order $s$. In the following, let us assume that every increasing induced path in $G$ has less than $s$ vertices (otherwise we are done).

Any two vertices $u \prec v$ of $G$ are connected by an increasing path, so they are connected by an increasing induced path (for instance any increasing path of minimum order). We define $P_{u,v}$ as an increasing induced path of \emph{maximum} order between $u$ and $v$. Note that by assumption on $G$, we have $|P_{u,v}|< s$.

\subsection{Coloring the triples}

Recall  that the hyperedges of the 3-clique on $V(G)$ are all the 3-elements subsets $\{u,v,w\}$ of vertices of $G$. These subsets $\{u,v,w\}$ are in one-to-one correspondence with the ordered triples of vertices $(u,v,w)$ of $G$ (by the definition above, such a triple satisfies  $u\prec v \prec w$), and we will consider the two objects interchangeably. More precisely we will color ordered triples of vertices of $G$, and view the resulting coloring as a coloring of the 3-clique on $V(G)$. We will then discuss the properties of a monochromatic 3-clique $K$ that can be found there. If $h$ is a function whose domain is the set of ordered triples of $G$, we say that \emph{$K$ is a monochromatic 3-clique for $h$} if $K$ is a monochromatic 3-clique in the 3-clique on $V(G)$ whose hyperedges have been colored according to $h$.

\medskip

We now define our first coloring $\col_1$ of the ordered triples of $V(G)$. For any ordered triple $(u, v, w)$ of $V(G)$, we set
\[
\col_1(u, v, w) = \begin{cases}
			1, & \text{if } |P_{u,v}| > |P_{u,w}|\\
            2, & \text{if } |P_{u,v}| < |P_{u,w}|\\
			3, & \text{if } |P_{u,v}| = |P_{u,w}| \text{ and } |P_{v,w}| > |P_{u,w}|\\
            4, & \text{if } |P_{u,v}| = |P_{u,w}| \text{ and } |P_{v,w}| < |P_{u,w}|\\
            0, & \text{otherwise.}
		 \end{cases}
\]

\begin{lemma}\label{lem:sizes}
Let $K$ be a monochromatic 3-clique for $\col_1$ of size at least $s$, then for any ordered triple $(u, v, w)$ of $K$ we have $\col_1(u, v, w) = 0$. In particular $|P_{u, v}| = |P_{u, w}| = |P_{v, w}|$.
\end{lemma}
\begin{proof}
Let $c$ be the unique color of the ordered triples of $K$ in  $\col_1$, and assume for the sake of contradiction that $c\ne 0$.
We can assume without loss of generality, up to taking the reverse ordering, that $c\in\{1,2\}$.
Let $u_0 \prec u_1 \prec \dots \prec u_{s-1}$ be the vertices of $K$.

The sequence $(|P_{u_0, u_i}|)_{1 \le i \le s-1}$ is strictly monotone with $i$: it is strictly decreasing if $c=1$, and strictly increasing  if $c=2$.
Hence since the order of any path $P_{u_0, u_i}$ is at least 2, at least one of the paths $P_{u_0, u_1}$ and $P_{u_0, u_{s-1}}$ has order at least $2 + (s-2) = s$, a~contradiction.
\end{proof}

Lemma \ref{lem:sizes} immediately implies the following statement.

\begin{remark}\label{rmk:common-order}
Let $K$ be a monochromatic 3-clique for $\col_1$ of size at least $s$, then there exists an integer $\ell \in \intv{2}{s-1}$ such that for any vertices $u, v$ of $K$, $|P_{u, v}| = \ell$. 
\end{remark}

  Note that for an ordered triple $(u, v, w)$ of $V(G)$, $P_{u, v}$ and $P_{u, w}$ are not necessarily internally vertex-disjoint. To keep track of this fact, we denote by $x(u, v, w)$ the largest common vertex of $P_{u, v}$ and $P_{u, w}$, and we denote by $d(u, v, w)$ its distance to $u$ on the path $P_{u, v}$, and by $d'(u, v, w)$ its distance to $u$ on the path $P_{u, w}$. Note that, if $d(u, v, w) = 0$ then $x(u, v, w) = u$.
  
 \smallskip
 
 We now define a second coloring of the ordered triples of vertices of $G$.
 For every ordered triple $(u, v, w)$ of $V(G)$, we set $\col_2(u, v, w) = (d(u, v, w) , \delta(u, v, w))$, where:
\[
\delta(u, v, w) = \begin{cases}
			-1, & \text{if } d(u,v,w) < d'(u, v, w)\\
            0, & \text{if } d(u,v,w) = d'(u, v, w)\\
            1, & \text{if } d(u,v,w) > d'(u, v, w).
		 \end{cases}
\]
 
\begin{lemma}\label{lem:split}
Let $K$ be a monochromatic 3-clique for $\col_2$ of size at least $s$. Then there exists an integer $d < s$ such that  for any ordered triple $(u, v, w)$ of $K$ we have $\col_2(u, v, w)= (d, 0)$.
\end{lemma}
\begin{proof}
As $K$ is monochromatic there are $d\in \intv{0}{s-1}$ and $\delta \in \{-1, 0, 1\}$ such that for any ordered triple $(u, v, w)$ of $K$, $\col_2(u, v, w)=(d, \delta)$. Towards a contradiction, suppose that $\delta \in \{-1, 1\}$.
Let $u_0 \prec u_1 \prec \dots \prec u_{s-1}$ be the vertices of $K$.

Consider, for any $0 < i < j < s$, the vertex $x_{i, j} := x(u_0, u_{i}, u_{j})$.
By definition $x_{i, j}$ is at distance $d(u_0, u_i, u_j) = d$ from $u_0$ along $P_{u_0, u_i}$ and at distance $d'(u, u_{i}, u_{j})$ from $u_0$ along $P_{u_0, u_{j}}$.
Note that since $d(u_0, u_i, u_j)$ does not depend on $j$ (as in $K$ it is equal to $d$), we have that $x_{i, j} = x_{i, j'}$ for any pair $j, j'$. In particular, in the case where $d = 0$, then any $x_{i, j}$ is the vertex $u_0$ and we reach an immediate contradiction with the value of $\delta$. In the following let $x_i = x_{i,j}$ for any $j>i$.

Consider three integers $i,j,k$ with  $0\le i < j < k<s$.
We know $x_{i}$ is at distance $d'(u, u_{i}, u_{j})$ from $u_0$ along $P_{u_0, u_j}$, and $x_{j}$ is at distance $d$ from $u_0$ along $P_{u_0, u_j}$.

\textbf{Case $\delta = -1$:}
We have $d < d'(u, u_{i}, u_{j})$, hence $x_{j} \prec x_{i}$.
Since $d>0$ we obtain $u_0 \prec x_{s-2} \prec \dots \prec x_{1} \prec u_{s-1}$. All these vertices lie on the path $P_{u_0, u_{s-1}}$, which is thus of order at least $s$, a~contradiction.

\textbf{Case $\delta = 1$:}
We have $d > d'(u, u_{i}, u_{j})$, hence $x_{i} \prec x_{j}$.
By the same argument we have $u_0 \prec x_{1} \prec \dots \prec x_{s-2} \prec u_{s-1}$. All these vertices lie on the path $P_{u_0, u_{s-1}}$, which is thus of order at least $s$, a~contradiction.
\end{proof}

Symmetrically we define $x^{\rev}(u,v,w)$ as the smallest common vertex of $P_{u,w}$ and $P_{v,w}$ and $d^\rev(u,v,w)$ (resp.\ $d'^\rev(u,v,w)$) as the distance between this vertex and $w$ on $P_{v,w}$ (resp. $P_{u,w}$). 
For every ordered triple $(u, v, w)$ of $V(G)$, let \[\col_2^\rev(u, v, w) = (d^\rev(u, v, w) , \delta^\rev(u, v, w)),\] where $\delta^\rev$ is defined analogously as $\delta$:\[
\delta^\rev(u, v, w) = \begin{cases}
			-1, & \text{if } d^\rev(u,v,w) < d'^\rev(u, v, w)\\
            0, & \text{if } d^\rev(u,v,w) = d'^\rev(u, v, w)\\
            1, & \text{if } d^\rev(u,v,w) > d'^\rev(u, v, w).
		 \end{cases}
\]
A proof symmetric to that of \Cref{lem:split} yields the following statement.

\begin{lemma}\label{lem:split-rev}
If $K$ is a monochromatic 3-clique for $\col_2^\rev$ of order at least $s$, then there exists an integer $d^\rev < s$ such that such that  for any ordered triple $(u, v, w)$ of $K$ we have $\col_2^\rev(u,v,w) = (d^\rev, 0)$.
\end{lemma}

\begin{remark}\label{rem:col2}
Since the paths $P_{u,v}$ have order less than $s$, $\col_2$ and $\col_2^\rev$ each take at most $3s$ values.
\end{remark}

In our proofs below it will be convenient to discard the largest and smallest vertices of a monochromatic 3-clique $K$, so we define the \emph{interior} $K^\circ$ of $K$ as $K$ minus its largest and  smallest vertices.

\begin{lemma}\label{lem:u+-}
Let $K$ be a monochromatic 3-clique of order at least $s$ for $\col_1 \times \col_2 \times \col_2^{\rev}$. Let $(0, (d , 0) , (d^{\rev} , 0))$ be the unique color\footnote{The fact that  the colors of $\col_2$ and $\col_2^{\rev}$ are of this form is a direct consequence of \Cref{lem:split} and \Cref{lem:split-rev}.} of all hyperedges of $K$.
Then for any vertex $v \in V(K^\circ)$ there are two vertices $v^-,v^+\in V(G)$ such that the following holds:
\begin{enumerate}
    \item\label{it:+-order} for any $v\in K^\circ$,  $v^- \preceq v \preceq v^+$,
    \item\label{it:com_vertex+} for any  $u \prec v \prec w$ in $K^\circ$, we have that $u^+ = x(u, v, w)$, i.e., $u^+$ is the largest common vertex of $P_{u,v}$ and $P_{u, w}$ (and it lies at distance $d$ from $u$ on both paths), 
    \item\label{it:com_vertex-} for any  $u \prec v \prec w$ in $K^\circ$ we have that $w^- = x^\rev(u,v,w)$, i.e., $w^-$ is the smallest common vertex of $P_{v,w}$ and $P_{u, w}$ (and it lies at distance $d^\rev$ from $w$ on both paths),  
    \item\label{it:ordering} for any $u \prec v$ in $K^\circ$, we have $u \preceq u^+ \prec v^- \preceq v$. In particular, $\ell > d + d^{\rev} + 1$.
\end{enumerate}
\end{lemma}
\begin{proof}
Fix any vertex $v$ in $K^\circ$, and consider a vertex $u$ of $K$ smaller than $v$, and a vertex $w$ of $K$ larger than $v$. 
We define $v^+$ as the vertex at distance $d$ from $v$ along the path $P_{v, w}$ and $v^-$ as the vertex at distance $d^{\rev}$ from $v$ along the path $P_{u, v}$.
The fact that $v^-$ and $v^+$ do not depend on the particular choices of $u$ and $w$ directly follows from the fact that $K$ is monochromatic of color $(0, (d , 0) , (d^{\rev} , 0))$. Item \eqref{it:+-order} follows from the fact that for any $u\prec v$, $P_{u,v}$ is an increasing path (with respect to $\prec$), and thus $u \preceq u^+$ and $v^-\preceq v$. Items \eqref{it:com_vertex+} and \eqref{it:com_vertex-} are direct consequences of \Cref{lem:split,lem:split-rev}. 

It remains to prove \eqref{it:ordering}. By \Cref{rmk:common-order}, there is an integer $\ell \in \intv{2}{s-1}$ such that for any $u,v\in K$, $|P_{uv}|=\ell$. To prove \eqref{it:ordering}, it suffices to show that $d + d^{\rev} < \ell-1$.
    Let $u \prec v \prec w$ be three vertices of $K$. Recall that $|P_{u, v}|=|P_{u, w}|=\ell$.
    We know that $u^+$ is contained in $P_{u, v}$ and thus $u^+ \preceq v$, and symmetrically $v \preceq w^-$.
    As $P_{u, w}$ contains $d$ vertices before $u^+$ and $d^{\rev}$ vertices after $w^-$, we obtain $d + d^{\rev}  \leq \ell-1$.
    If equality holds, then  $u^+ = w^-$, and thus $u^+ = w^- = v$. But then the distance from $u$ to  $u^+$ in $P_{u,v}$ is $\ell-1$, so $d=\ell-1$, and symmetrically $d^\rev=\ell-1$. We obtain $2\ell-2=\ell-1$ and thus $\ell=1$, which is a contradiction since $P_{u,v}$ contains at least two vertices ($u$ and $v$).
\end{proof}

\begin{lemma}\label{lem:d+-l-2}
Let $K$ be a monochromatic 3-clique  for $\col_1 \times \col_2 \times  \col_2^{\rev}$ of order at least $2t + 2$. Let $(0, (d , 0) , (d^{\rev} , 0))$ be the unique color of all hyperedges of $K$. If $d + d^{\rev} = \ell - 2$, then $G$ contains a $K_{t,t}$ subgraph.
\end{lemma}
\begin{proof}
Recall that for every $u, v \in V(K^\circ)$, $|P_{u,v}| = \ell$. Therefore if $d + d^{\rev} = \ell - 2$ then $u^+v^-$ is an edge for every $u \prec v \in V(K^\circ)$.
Let $u_1, \dots, u_t, v_1, \dots, v_t$ be vertices of $K^\circ$ in this order with respect to $\prec$. Then for every $i,j \in \intv{1}{t}$, $u_i^+ v_j^-$ is an edge. This is the $K_{t,t}$ subgraph we were looking for. 
\end{proof}

Given a pair $d,d^\rev\ge 0$ of integers and a pair $u\prec v$ of vertices of $G$ such that $|P_{u,v}| >  d+d^\rev+1$, we denote by $\Pleft_{u,v}(d,d^\rev)$ the subpath of $P_{u,v}$ starting in $u$ and ending in the vertex of $P_{u,v}$ at distance $d$ from $u$ on this path. 
Symmetrically $\Pright_{u,v}(d,d^\rev)$ is the subpath of $P_{u,v}$ starting at the vertex at distance $d^\rev$ from $v$ on this path,  and ending in~$v$.
The vertices of $P_{u,v}$ not belonging to any of these two subpaths form a subpath we call $\Pmid_{u,v}(d,d^\rev)$.
When the pair $(d,d^\rev)$ is clear from the context, we simply write $\Pleft_{u,v}$, $\Pright_{u,v}$ and $\Pmid_{u,v}$. We refer to \Cref{fig:pathP} for a depiction of these paths.

\medskip

For some  monochromatic 3-clique $K$ of order at least $s$ with respect to  the coloring $\col_1 \times \col_2 \times \col_2^{\rev}$, we say that the \emph{type of $K$} is $(\ell, d, d^\rev)$ when $(0, (d, 0), (d^\rev, 0))$ is the unique color of $K$ and $\ell$ is the order of all paths $P_{u,v}$ with $u,v\in K$ (see \Cref{rmk:common-order}).

\smallskip

Let $K$ be of type $(\ell, d, d^\rev)$. 
Observe that by Lemma \ref{lem:u+-}, for every $u\in K^\circ$, the paths of the form $\Pleft_{u,v}(d,d^\rev)$ (for $v \in K^\circ$ with $u\prec v$) all start in $u$ and end at the same vertex $u^+$ (and symmetrically for all paths of the form $\Pright_{u,v}(d,d^\rev)$).
For any $v$ in $K^\circ$, we define the \emph{territory} $\ter(v, K^\circ)$ of $v$ as follows:
\[
\ter(v,K^\circ) = %
\bigcup_{\substack{u \in K^\circ\\\text{s.t.}\ u\prec v}} \Pright_{u, v}(d,d^\rev) \cup %
\bigcup_{\substack{w\in K^\circ\\\text{s.t.}\ v \prec w}} \Pleft_{v, w}(d,d^\rev).
\]
We simply write $\ter(v)$ when  $K^\circ$ is clear from the context.

\begin{corollary}\label{cor:sum}
Let $K$ be a  monochromatic 3-clique  for $\col_1 \times \col_2 \times \col_2^{\rev}$ of size at least $s$, and let $K^\circ$ be its interior. We have the following properties:
\begin{enumerate}
    \item\label{it:ter_disjoint} Distinct vertices of $K^\circ$ have vertex-disjoint territories,
    \item\label{it:after_u+} for any $u\prec v \prec w$ in $K^\circ$,  $\Pmid_{u, v} \cup \Pright_{u, v}$ is vertex-disjoint from $\Pmid_{u, w} \cup \Pright_{u, w}$, and
    \item\label{it:before_u-} for any $u\prec v \prec w$ in $K^\circ$,  $\Pleft_{u, w} \cup \Pmid_{u, w}$ is vertex-disjoint from $\Pleft_{v, w} \cup \Pmid_{v, w}$.
\end{enumerate}
\end{corollary}
\begin{proof}\mbox{}
\begin{itemize}
    \item[\eqref{it:ter_disjoint}] Let $u\prec v$ in $K^\circ$.
By Lemma~\ref*{lem:u+-}\eqref{it:+-order}, $u^- \preceq u\preceq u^+$.
As all the paths we consider are increasing, every vertex $z$ of a path in the definition of $\ter(u)$ is such that $u^-\preceq z \preceq u^+$ and similarly for $\ter(v)$. As $u^+\prec v^-$ (by Lemma~\ref*{lem:u+-}\eqref{it:ordering}), $\ter(u) \cap \ter(v) = \emptyset$.
\item[\eqref{it:after_u+}] Let $u\prec v\prec w$ in $K^\circ$. By definition of $u^+$ and $v^-$, $\Pmid_{u,v}$ is vertex-disjoint from $P_{u,w}$, and symmetrically for $\Pmid_{u,w}$ and $P_{u,v}$. Moreover, $\Pright_{u,v}$ is part of $\ter(v)$ and $\Pright_{u,w}$ is part of $\ter(w)$ so by the previous item, they are disjoint.
\item[\eqref{it:before_u-}] Symmetric to the proof of the previous item.
\end{itemize}
\end{proof}

We now define a new coloring $\col_3$ of the ordered triples of vertices of $G$ as follows. Consider an ordered triple $(u,v,w)$ of $V(G)$ and assume its color with respect to $\col_1 \times \col_2 \times \col_2^{\rev}$ is $(c_1, (d , \delta) , (d^\rev , \delta^\rev))$. If $c_1 \ne 0$ or $\delta \ne 0$ or $\delta^\rev \ne 0$ we set $\col_3(u,v,w)=(0,0,0)$. Otherwise we can assume that $c_1 = \delta = \delta^\rev = 0$. Since $c_1 = 0$, there is an integer $\ell \in \intv{2}{s-1}$ such that $\ell = |P_{u, v}| = |P_{u, w}| = |P_{v, w}|$ (\Cref{rmk:common-order}).
If $\ell<d+d^\rev+2$, then we again set $\col_3(u,v,w)=(0,0,0)$. Otherwise $\ell \ge d+d^\rev+2$ and for all $x, y \in \{u,v,w\}$ with $x\prec y$, the paths $\Pleft_{x,y}=\Pleft_{x,y}(d,d^\rev)$, $\Pmid_{x,y}=\Pmid_{x,y}(d,d^\rev)$, and  $\Pright_{x,y}=\Pright_{x,y}(d,d^\rev)$ are well defined. 
In this case we set $\col_3(u,v,w) = (i, a, b)$ where
$$i=\begin{cases}
			1, & \text{if there is an edge $pq\not\in P_{u,v}$ between $\Pleft_{u,w}$ and $\Pmid_{u,v} \cup \Pright_{u, v}$}\\
            2, & \text{otherwise, if there is an edge $pq\not\in P_{v,w}$ between $\Pleft_{v, w} \cup \Pmid_{v,w}$ and $\Pright_{u,w}$}\\
            3, &\text{otherwise, if there is an edge $pq$ between $\Pleft_{u,w}$ and $\Pleft_{v,w}\cup\Pmid_{v,w}$}\\
            4, &\text{otherwise, if there is an edge $pq$ between $\Pmid_{u,v}\cup \Pright_{u,v}$ and $\Pright_{u,w}$}\\
            5, &\text{otherwise, if there is an edge $pq$ between $\Pmid_{u,v}$ and $\Pmid_{v,w}$}\\
            0, & \text{otherwise.}
		 \end{cases}$$

If $i=0$ we set $a=b=0$. In the other cases, $a$ and $b$ are the respective indices of $p$ and $q$ on the considered paths, assuming $p\prec q$. For instance if $i=2$, then $a$ is the index of $p$ on $\Pleft_{v,w}\cup \Pmid_{v,w}$ and $b$ is the index of $q$ on $\Pright_{u,w}$. The condition that $pq\not\in P_{v,w}$ in the case $i=2$ is there to exclude the edge of $P_{v,w}$ which connects the last vertex of $\Pleft_{v,w}\cup \Pmid_{v,w}$ to the first vertex of $\Pright_{u,w}$ (this edge is always present). We emphasize here that $i=0$ does not mean that there are no edges between the three paths $P_{u,v}$, $P_{u,w}$ and $P_{v,w}$, simply that such edges (if they exist) do not follow the patterns described above. See \Cref{fig:6P} for an illustration of the possible adjacencies between the different subpaths appearing in the definition of $i$ above.

\begin{remark}\label{rem:col3}
The paths $P_{u,v}$ have order less than $s$, so $\col_3$ takes at most $5s^2 +1$ values.
\end{remark}

The next step will be to apply Ramsey's theorem for triples (Theorem \ref{th:ramseyc} with $k=3$) to the coloring $\col_1 \times \col_2 \times \col_2^\rev \times \col_3$. Assume that in the resulting monochromatic 3-clique, the color of  $\col_3$ is  $(i,a,b)$, for some indices $i\ne 0$ and $a,b$. Then for any ordered triple $(u,v,w)$ of vertices of this 3-clique, there is an edge between  one of the 3 paths relating $u,v,w$, and  another of these paths (the endpoints of this edge are denoted by $f(u,v,w)$ and $f'(u,v,w)$ in Lemma \ref{lem:edges123} below). There is a little bit of a case analysis depending on which subpaths are connected by an edge, but the cases are mostly symmetric.  We will use the next lemma  to find a copy of $K_{t,t}$ whenever $i\ne 0$. 

\medskip

Given a subset $V'$ of vertices of $G$, a function $f$ defined on ordered triples of $V'$, and two fixed vertices $v\prec w$ of $V'$, we define $f(\cdot,v,w)$ as the function $u\mapsto f(u,v,w)$, with domain $\{u\in V': u\prec v\}$. Similarly, for $u\prec w$, $f(u,\cdot,w)$ is the function $v\mapsto f(u,v,w)$, with domain $\{v\in V': u\prec v\prec w\}$. The function $f(u,v,\cdot)$ is defined similarly.

\begin{lemma}\label{lem:edges123}
Let $V'$ be a subset of $V(G)$ of size at least $2t+3$.
Let $f$ and $f'$ be two functions
mapping every ordered triple $(u_1, u_2, u_3)$ of $V'$ to a vertex of $P_{u_1, u_2} \cup P_{u_2, u_3} \cup P_{u_1, u_3}$ such that $f(u_1, u_2, u_3)f'(u_1, u_2, u_3)$ is an edge of $G$. If for all $u\prec v$ in $V'$, one of the following three items holds
\begin{enumerate}
\item \label{it:1}
\begin{enumerate}
    \item \label{it:sigma1a} $f(\cdot, u, v)$ is injective and  $f'(\cdot, u, v)$ is constant, and 
    \item \label{it:sigma2a} $f(u, \cdot , v)$ is constant and $ f'(u, \cdot , v)$ is injective;
\end{enumerate}
\item
\begin{enumerate}
    \item \label{it:sigma1b} $f(u, v,\cdot)$ is injective and  $f'(u, v,\cdot)$ is constant, and 
    \item \label{it:sigma2b} $f(u, \cdot , v)$ is constant and $ f'(u, \cdot , v)$ is injective,
\end{enumerate}
\item
\begin{enumerate}
    \item \label{it:sigma1c} $f(\cdot,u, v)$ is injective and  $f'(\cdot,u, v)$ is constant, and 
    \item \label{it:sigma2c} $f(u, v, \cdot)$ is constant and $ f'(u, v,\cdot)$ is injective,
\end{enumerate}
\end{enumerate}
then $G[V']$ contains a $K_{t,t}$ subgraph.
\end{lemma}

\begin{proof}
We give the proof for the case where \eqref{it:1} holds.
We consider $2t+3$ elements $X=\{x_i:1\le i \le t\}$, $Y=\{y_i:1\le i \le t\}$, and $Z=\{z_1,z_2,z_3\}$ of $V'$ ordered as follows:
\[
z_1 \prec x_1 \prec \dots \prec x_{t} \prec z_2 \prec y_1 \prec \dots \prec y_{t} \prec z_3. 
\]
 
 Observe that for any $x\in X$ and $y\in Y$, we have $x\prec y \prec z_3$.

By assumption \eqref{it:sigma1a}, for every fixed $y\in Y$, the values taken by $x \mapsto f(x, y, z_3)$ with $x\in X$ are all distinct. For every $i,j\in [t]$ we call $a_{i,j}$ the value taken by $f(x_i, y_j, z_3)$. Observe also that $x \mapsto f'(x, y_j, z_3)$ is constant, equal to some vertex we call $b_j$.

By assumption \eqref{it:sigma2a}, for every fixed $x\in X$ the function $y \mapsto f(x, y, z_3)$ is constant for $y\in Y$ so actually $a_{i,j} = a_{i,j'}$ for every $j,j'\in[t]$. So we simply call this vertex $a_i$ from now on.
Moreover, $y \mapsto f'(x, y, z_3)$ is injective for any fixed $x$, so the vertices $b_j$'s defined above are all distinct.

We have obtained two $t$-element subsets $\{a_1, \dots, a_t\}$ and $\{b_1, \dots, b_t\}$ of  vertices of $G$  and by our assumptions on $f$ and $f'$, $a_ib_j$ is an edge for every $i,j\in [t]$ (so in particular these sets are disjoint). It follows that $G[V']$ contains $K_{t,t}$ as a subgraph.

The two other cases have nearly identical proofs, replacing only $z_3$ by $z_2$ or $z_1$ depending on which vertex in the ordered triple is fixed.
\end{proof}

We deduce the following.

\begin{lemma}\label{lem:no_edge}
Let $K$ be  a monochromatic 3-clique for $\col_1 \times \col_2 \times \col_2^\rev \times \col_3$. If $K$ has order at least $\max(s, 2t+5)$ then $G$ has a $K_{t, t}$ subgraph or $\col_3$ is equal to (0,0,0) on $K$.
\end{lemma}

\begin{proof}
Assume that $K$ has order at least $\max(s, 2t+5)$, and thus its interior  $K^\circ$ of $K$ has order at least $2t+3$.
 Let $(i,a,b)$ be the (unique) value of $\col_3$ on $K$. If $i=0$ then by definition of $\col_3$ we also have $a=b=0$ and we are done. So in the following we assume that $i\neq 0$. We only need to show that in this case $G$ contains a $K_{t,t}$ subgraph.

 Let $(u_1, u_2, u_3)$ be an ordered triple of $K^\circ$. Recall that in order to define $\col_3(u_1, u_2, u_3)$ in the current case where $i\neq 0$, we chose an edge $xy$ between some vertex $x$ on one of $P_{u_1, u_2}$, $P_{u_2, u_3}$, and $P_{u_1, u_3}$, and a vertex $y$ lying on another of these three paths, and the positions of $x$ and $y$ were then used to set the values of $a$ and $b$. We now set $f(u_1, u_2, u_3) = x$ and $f'(u_1, u_2, u_3)=y$. We consider the following cases:
 
 \medskip
 
\noindent\underline{$\boldsymbol{i=1}$} ($x\in \Pleft_{u_1,u_3}$, $y\in \Pmid_{u_1,u_2} \cup \Pright_{u_1, u_2}$ and $xy\not\in P_{u_1,u_2}$).
 As $P_{u_1,u_2}$ is independent of $u_3$, $f'(u_1,u_2,\cdot)$ is constant. Similarly, as $P_{u_1,u_3}$ is independent of $u_2$, $f(u_1,\cdot,u_3)$ is also constant. Moreover, it follows from \Cref*{cor:sum}\eqref{it:after_u+} that for fixed $u_1$, the subpaths $\Pmid_{u_1,u_2} \cup \Pright_{u_1,u_2}$  are all vertex-disjoint, and thus $f'(u_1,\cdot,u_3)$ is injective.

 It remains to prove that $f(u_1,u_2,\cdot)$ is injective. Consider any two vertices $u_3, u_3'$ in $K^\circ$ with $u_2 \prec u_3 \prec u_3'$. Consider also the following vertices:
 \begin{eqnarray*}
      \alpha=f(u_1,u_2,u_3)\in \Pleft_{u_1,u_3}  \\
      \alpha'=f(u_1,u_2,u_3')=f(u_1,u_3,u_3')\in \Pleft_{u_1,u_3'}\\ 
      \beta=f'(u_1,u_3,u_3') \in \Pmid_{u_1,u_3} \cup \Pright_{u_1, u_3}
 \end{eqnarray*}

 As $(u_1, u_3, u_3')$ is an ordered triple of $K^\circ$,  $\col_3(u_1, u_3, u_3')=(1,a,b)$ and thus there is an edge between $\alpha'$ and $\beta$ in $G$.
 As $P_{u_1, u_3}$ is an induced path, and since $\alpha$ is not the predecessor of $\beta$ on the path $P_{u_1,u_3}$, $\alpha$ is not adjacent to $\beta$ in $G$. Hence, $\alpha \ne\alpha'$, and thus $f(u_1, u_2, u_3) \neq f(u_1, u_2, u_3')$. This shows that $f(u_1,u_2,\cdot)$ is injective, as desired.
 
 We can thus apply Lemma \ref{lem:edges123}, which implies that $G$ contains a $K_{t,t}$ subgraph.
 
 \medskip
 
 \noindent\underline{$\boldsymbol{i=2}$} ($x\in \Pleft_{u_2,u_3} \cup \Pmid_{u_2, u_3}$, $y\in \Pright_{u_1,u_3}$, and $xy\not\in P_{u_2,u_3}$). This is symmetric to the previous case.
 
 \medskip
 
 \noindent\underline{$\boldsymbol{i=3}$} ($x\in \Pleft_{u_1,u_3}$, $y\in \Pleft_{u_2,u_3}\cup \Pmid_{u_2,u_3}$). It follows from \Cref*{cor:sum}\eqref{it:ter_disjoint} that the territories $\ter(u_1)$ are disjoint for distinct values of $u_1$, which implies that $f(\cdot,u_2,u_3)$ is injective. As 
  $P_{u_2,u_3}$ does not depend on $u_{1}$, $f'(\cdot,u_2,u_3)$ is constant and similarly, as 
$P_{u_1,u_3}$ does not depend on $u_2$, $f(u_1,\cdot,u_3)$ is constant. Finally, by \Cref*{cor:sum}\eqref{it:before_u-}, for fixed $u_3$ the subpaths $\Pleft_{u_2,u_3}\cup\Pmid_{u_2,u_3}$ are disjoint for distinct values of $u_2$, and thus $f'(u_1,\cdot,u_3)$ is injective.

We can thus apply Lemma \ref{lem:edges123}, which implies that $G$ contains a $K_{t,t}$ subgraph.

\medskip

\noindent\underline{$\boldsymbol{i=4}$} ($x\in \Pmid_{u_1,u_2}\cup \Pright_{u_1,u_2}$, $y\in \Pright_{u_1,u_3}$). This is symmetric to the previous case (we only need to swap the role of $f$ and $f'$, and use \Cref*{cor:sum}\eqref{it:after_u+} instead of \Cref*{cor:sum}\eqref{it:before_u-}).

\medskip

\noindent\underline{$\boldsymbol{i=5}$} ($x\in \Pmid_{u_1,u_2}$, $y\in \Pmid_{u_2,u_3}$). By \Cref*{cor:sum}\eqref{it:before_u-}, for fixed $u_2$ the subpaths $\Pmid_{u_1,u_2}$ are disjoint for distinct values of $u_1$, and thus $f(\cdot,u_2,u_3)$ is injective. As  $P_{u_2,u_3}$ does not depend on $u_{1}$, $f'(\cdot,u_2,u_3)$ is constant and similarly, as $P_{u_1,u_2}$ does not depend on $u_{3}$, $f(u_1,u_2,\cdot)$ is constant. Finally, \Cref*{cor:sum}\eqref{it:after_u+} implies that for fixed $u_2$, the subpaths $\Pmid_{u_2,u_3}$ are disjoint for distinct values of $u_3$, and thus $f(u_1,u_2,\cdot)$ is injective. 

We can thus apply Lemma \ref{lem:edges123}, which implies that $G$ contains a $K_{t,t}$ subgraph.
\end{proof}

\subsection{Concluding the proof}
We are now ready to complete the proof of \Cref{th:main}.
We consider the following coloring of the ordered triples of $V(G)$: \[\col = \col_1 \times \col_2\times \col_2^\rev \times \col_3.\] According to the definition of $\col_1$ and Remarks~\ref{rem:col2} and~\ref{rem:col3}, $\col$ takes at most $5 \cdot (3s)^2 \cdot (5s^2+1) \le 10^3 s^4$ values (colors). Recall that \[n\geq  (10^3 s^4) \pow (10^3 s^4) \pow (2 \cdot 10^3 s^4 (\max(s, 2t + 5))).\]
By \Cref{th:ramseyc} (with $k=3$), there is a monochromatic clique $K$ for $\col$ of order $\max(s, 2t + 5)\ge s$, and the interior $K^\circ$ of $K$ has order at least $2t + 3$.
Assume for the sake of contradiction that $G$ has no $K_{t,t}$ subgraph.
By Lemmas~\ref{lem:split}, \ref{lem:split-rev}, and~\ref{lem:no_edge}, the unique color given to the triples of $K$ is of the form
$(0, (d, 0), (d^{\rev}, 0), (0,0,0))$, in particular $K$ has type $(\ell, d, d^\rev)$ for some~$\ell\in \intv{2}{s-1}$.

Consider any ordered triple $(u, v, w)$ of $K^\circ$, and consider the path \[
P = \Pleft_{u, w}\cdot \Pmid_{u, v} \cdot \Pright_{u,v}\quad \cdot\quad \Pleft_{v,w} \cdot \Pmid_{v, w} \cdot \Pright_{u, w},
\]
where the symbol $\cdot$ denotes the concatenation\footnote{It is convenient to for us to define $P_1 \cdot P_2$ when the last vertex of $P_1$ is adjacent to the first vertex of $P_2$, but also when the last vertex of $P_1$ coincides with the first vertex of $P_2$. These are ``vertex'' concatenation and ``edge'' concatenation in some sense, but we call both operations concatenation for simplicity.} of paths. The path $P$ is illustrated in Figure~\ref{fig:pathP}.

\begin{figure}[htb]
  \centering
    \includegraphics[scale=1.1]{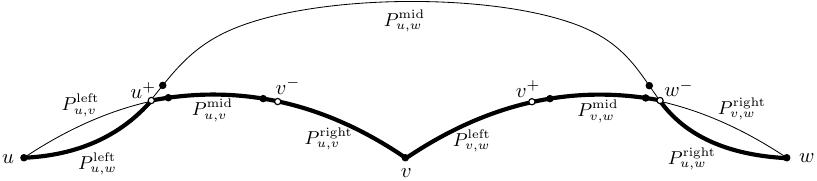}
  \caption{The path $P$ (in bold).\label{fig:pathP}}
\end{figure}

Note that $\Pleft_{u, w}\cdot \Pmid_{u, v} \cdot \Pright_{u,v}$ has the same order as $P_{u,v}$ (namely, $|P_{u,v}|=\ell$), and $\Pleft_{v,w} \cdot \Pmid_{v, w} \cdot \Pright_{u, w}$ has the same order as $P_{v,w}$ (namely, $|P_{v,w}|=\ell$). It follows that $P$ contains $2\ell-1>\ell$ vertices. Hence, the path $P$ is an increasing path from $u$ to $w$, with order greater than $\ell = |P_{u,w}|$ so it cannot be an induced path, by maximality of $P_{u,w}$.
So we know that $P$ has \emph{shortcuts} (by which we mean  edges of $G[P]$ that are not  edges of $P$).

\begin{figure}[htb]
  \centering
    \includegraphics[scale=0.8]{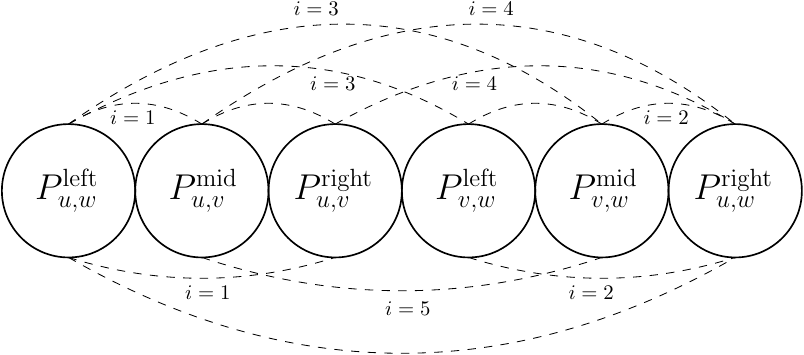}
  \caption{Adjacencies between the subpaths of $P$ (indicated by dashed lines) described by $\col_3$. This excludes the edges of $P$ connecting each subpath to the next subpath.\label{fig:6P}}
\end{figure}

Because the value of $\col_3$ on the triples of $K$ is $(0,0,0)$, we know that  shortcuts $xy$ (with $x\prec y$) on $P$ are severely restricted. It follows from the definition of $\col_3$ that there are no edges between certain subpaths of $P$.
Also, Lemma~\ref{lem:d+-l-2} implies that for every $u\prec w$ or $K$, the paths $\Pleft_{u,w}$ and $\Pright_{u,w}$ are not adjacent because $u^+w^-$ is not an edge and $P_{u,w}$ is induced.
The non-adjacencies are depicted by dashed lines in Figure \ref{fig:6P} (where we use both the fact that the first entry of $\col_3$ is 0 and the fact that the paths $P_{u,v}$, $P_{u,w}$ and $P_{v,w}$ are induced). In particular there are no shortcuts involving $v$ and no shortcuts between $\Pleft_{u, w}$ and $\Pmid_{u, v} \cup \Pright_{u, v}$ (and symmetrically between $\Pleft_{v, w}\cup \Pmid_{v, w}$ and $\Pright_{u, w}$).
Those two facts combined imply $x \prec v \prec y$.

Consider a shortcut $xy$ on $P$ with $x$ minimum and, subject to this, $y$ maximum. Let $Q$ be the path obtained by following $P$ from $u$ to $x$, then the edge $xy$, and then $P$ from $y$ to $w$. Note that $Q$ is an increasing induced  path from $u$ to $w$ (here we crucially use the property that $x \prec v \prec y$ for all shortcuts $xy)$.
By the definition of $\col_3$, and the aforementioned non-adjacencies between the different subpaths of $P$, we know that at least one of the following holds:
\begin{itemize}
    \item $u^+ \preceq x \prec v$ and $v\prec y \preceq v^+$. In this case $Q$ is the concatenation of $\Pleft_{u,w}$, a path of order at least 2 between $u^+$ and $v^+$, and $\Pmid_{v,w} \cdot \Pright_{u,w}$. It follows that $Q$ has order at least $\ell +1>|P_{u,w}|$, contradicting the maximality of $P_{u,w}$, or
    \item $v^- \preceq x\prec v$ and $v^+\prec y \preceq w^-$, in which case $Q$ is the concatenation of $\Pleft_{u,w} \cdot \Pmid_{u,v}$, a path of order at least 2 between $v^-$ and $w^-$, and $\Pright_{u,w}$. It follows that $Q$ has order at least  $\ell +1>|P_{u,w}|$, which is again a contradiction.
\end{itemize}

We reached the desired contradiction, so $G$ contains an increasing induced path of order~$s$ or a $K_{t,t}$ subgraph. This concludes the proof of \Cref{th:main}.\qed

\section{Forbidden patterns}\label{sec:pattern}

\subsection{Preliminaries}

An \emph{ordered graph} is a graph with a total order on its
vertex set. Consider an ordered graph $G$ with order
$v_1,\ldots,v_n$ and an ordered graph $H$ with order
$u_1,\ldots,u_k$. We say that $G$ \emph{contains $H$ as an ordered subgraph}
if
there exist $1\le a_1<a_2<\ldots <a_k\le n$ such that for all $1\le
i,j\le k$, if $u_{i}$ is adjacent to $u_{j}$ in $H$, then $v_{a_i}$ is
adjacent to $v_{a_j}$ in $G$. In words, $H$ appears as a subgraph in
$G$ in such a way that  the ordering of the copy of $H$ in  $G$ is consistent with the
ordering of $G$.

 \medskip

In this section, we will usually think of ordered graphs as being ordered from left to right (with the smallest vertices on the left, and the largest on the right). This convention holds in all figures. We will sometimes also use it in expressions such as \say{all vertices to the left of $u$} or \say{$v$ has no neighbor to the right}, by which we mean \say{all vertices smaller than $u$} and \say{$v$ has no neighbor larger than itself}, respectively.

\medskip

Let $G$ be a graph and $P=v_1,\ldots,v_n$ be a Hamiltonian path in
$G$. Note that $P$ allows us to consider $G$ and $G-E(P)$ (the spanning
subgraph of $G$ obtained by removing the edges of $P$) as ordered
graphs. Given an ordered graph $H$, we say that \emph{$(G,P)$ contains $H$ as a
pattern} if the ordered graph $G-E(P)$ contains $H$ as an ordered
subgraph. If $(G,P)$ does not contain $H$ as a pattern, we say that \emph{$(G,P)$
  avoids the pattern $H$}. When $P$ is clear from the context we simply say that $G$ contains or avoids the pattern $H$ (but in all such instances we really mean that $H$ is a pattern with respect to some Hamiltonian path $P$, so the edges of $P$ are not part of the pattern).

\medskip

For any ordered graph $H$, and integer $n$, we let $g_H(n)$ be the
maximum integer $t$ such that for any graph $G$ with a Hamiltonian path
$P=v_1,\ldots,v_n$, if $(G,P)$ avoids the pattern $H$ then $G$
contains an induced path of order at least $t$.

\medskip

Our main objective is to classify the growth rate of $g_H$, depending of the
structure of the 
ordered graph $H$. We start with some basic observations.

\begin{observation}\label{obs:order} If $G$ is a graph with a Hamiltonian path
  $v_1,\ldots,v_n$ such that for every edge $v_iv_j$ in $G$,
  $|i-j|\le \ell$, then $G$ contains an induced path of order at
  least $n/\ell$.
\end{observation}

\begin{proof}
It suffices to consider a shortest path from $v_1$ to $v_n$ in
$G$. This is an induced path in $G$ of order at least $(n-1)/\ell+1\ge n/\ell$.
\end{proof}

Let $A$ and $B$ be two ordered graphs. The \emph{concatenation} of $A$
and $B$, denoted by $A\cdot B$, is the ordered graph whose graph is the
disjoint union of $A$ and $B$, and whose order consists of the ordered
vertices of $A$, followed by the ordered vertices of $B$.

\begin{lemma}\label{lem:concat}
Let $A$ and $B$ be two ordered graphs. Then for any $n\ge 0$,
\[g_{A\cdot B}(n)\ge \min\{g_A(\lfloor n/2\rfloor),g_B(\lceil n/2\rceil)\}.\]
\end{lemma}

\begin{proof}
  Let $G$ be a graph with a Hamiltonian path $P=v_1,\ldots,v_n$. We
  divide $P$ into two paths $P_1$ (on $\lfloor n/2\rfloor$ vertices)
  and $P_2$ (on $\lceil n/2\rceil$ vertices). For $i\in \{1,2\}$, let
  $G_i$ be  the ordered subgraph
  of $G$ induced by $P_i$. If $(G_1,P_1)$ avoids the pattern $A$,
  then $G_1$ (and thus $G$) contains an induced path of order
  $g_A(\lfloor n/2\rfloor)$. If $(G_2,P_2)$ avoids the pattern $B$,
  then $G_2$ (and thus $G$) contains an induced path of order
  $g_B(\lceil n/2\rceil)$. Otherwise $(G_1,P_1)$ contains the pattern
  $A$ and $(G_2,P_2)$ contains the pattern $B$, and thus $(G,P)$
  contains the pattern $A\cdot B$. This shows that if $(G,P)$ avoids the
  pattern $A\cdot B$, then it contains an induced path of order
  $\min\{g_A(\lfloor n/2\rfloor),g_B(\lceil n/2\rceil)\}$, as desired.
\end{proof}

We now consider two illuminating examples of pairs $(G,P)$ where $P$
is a Hamiltonian path in $G$.

\medskip

\noindent {\bf Example 1.} Consider a path $P$ with vertex set $1,2,\ldots,n$ and for every pair
of integers $1\le i<j\le n$ such that $i$ is odd and $j$ is even, add
an edge between $i$ and $j$ (if the edge is not already contained in
$P$). Let $G$ be the resulting bipartite graph (note that $G$ is a supergraph of
the \emph{half-graph}). We first observe  that the maximum order of an induced path in $G$ is 4.

To see this, assume for the sake of contradiction that there exists an induced path $Q$ in $G$ of order 5 with vertices $a_1, \dots, a_5$. The path $Q$ contains either 3 odd or 3 even vertices. Without loss of generality, assume that $Q$ contains 3 odd vertices. The odd vertices must be $a_1, a_3$ and $a_5$ since $G$ is bipartite, and $a_2$ and $a_4$ are even. Let $a$ be the largest vertex of $Q$.
If $a$ is even, then by definition of $G$ it is adjacent to $a_1, a_3$ and $a_5$ and thus $a$ has degree 3 in $Q$, a contradiction.
Thus $a$ is either $a_1, a_3$ or $a_5$, and by symmetry, we only need to consider the cases $a=a_1$ and $a=a_3$.

If $a=a_3$, then $a_2$ and $a_4$ are two vertices smaller than $a_3$ and adjacent to it, a contradiction with the fact that $a_3$ is odd.
If $a=a_1$ then $a_2$ is necessarily adjacent to $a_1$ in $P$, and so $a_2 = a_1 - 1$.
Thus we have $a_2 > a_5$, which implies that $a_2$ adjacent to $a_5$, a contradiction.
Hence $G$ does not contain an induced path on $5$ vertices (and the induced path $2,1,4,5$ shows that $G$ indeed contains an induced path on $4$ vertices).

\medskip

An interesting property of the graph $G$ of Example 1 above (as an ordered graph,
with the order $1,2,\ldots,n$ given by $P$), is that for each vertex
$1\le i \le n$, the neighbors of $i$ in $G$ distinct from $i-1$ and $i+1$ are
either all to the left of $i$, or all to the right of $i$. 
This shows that $(G,P)$ avoids any pattern in which some vertex has a neighbor to the left and
a neighbor to the right (the simplest example of such a pattern is a
path $u,v,w$ with $u\prec v \prec w$). 

\smallskip

Example 1 immediately implies the following.

\begin{observation}\label{obs:d2a}
Let $H$ be an ordered graph such that $\{g_H(n): n\in \mathbb{N}\}$ is unbounded. Then for each vertex $v\in V(H)$, all neighbors of $v$ are 
smaller than  $v$, or all neighbors of $v$ are larger than $v$. In particular $H$ is bipartite. 
\end{observation}

\noindent {\bf Example 2.}
For every integer $i\ge 1$, we define inductively an ordered graph $G_i$ such that any two consecutive vertices in the order are adjacent, so the sequence of vertices in the order forms a Hamiltonian path $P_i$ of $G_i$. The graph $G_1$  consists of a single vertex. For any $i\ge 2$, the ordered graph $G_i$ is defined as follows: we consider three vertices $u,v,w$ with $u\prec v \prec w$ and we add edges $uv$ and $uw$. We then place an ordered copy $A$ of $G_{i-1}$ between $u$ and $v$, and an ordered copy $B$ of $G_{i-1}$ between $v$ and $w$. Finally we add an edge between  $u$ and the first vertex of $A$, an edge between the last vertex of $B$ and $w$, and we join $v$ to the last vertex of $A$ and the first vertex of $B$. See Figure \ref{fig:g2g3} for an illustration of $G_2$ and $G_3$.

\begin{figure}[htb]
  \centering
    \includegraphics[scale=1]{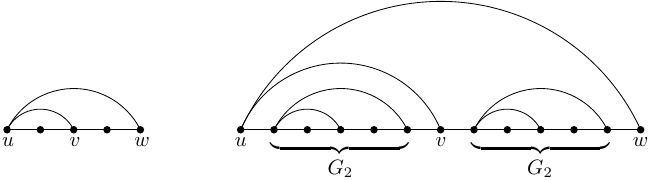}
  \caption{The ordered graphs $G_2$ (left), and $G_3$ (right), where vertices are ordered from
    left to right.\label{fig:g2g3}}
\end{figure}

The first property of $G_i$ is that each vertex $v$ has at most two neighbors preceding it in the order, including the neighbor of $v$ on the path $P_i$. In other words, each vertex of $G_i$ has at most one neighbor in the graph $G_i-E(P_i)$ that precedes it. 

The second property is that the longest induced path in $G_i$ has order $O(\log |V(G_i)|)$. To see why this holds, call an induced path in $G_i$ \emph{extremal} if one of its endpoints is $u$ or $w$, and 
observe that any extremal induced path in $G_{i}$  must consist of an extremal induced path in one of the two copies of $G_{i-1}$, together with at most 3 vertices of $G_{i}$ (namely $u$, $v$, or $w$), so a simple induction shows that extremal induced paths have order $O(\log |V(G_i)|)$. It remains to observe that any induced path in $G_i$  is either completely contained in a copy of $G_{i-1}$, or it consists of a union of two extremal induced paths, each in a different copy of $G_{i-1}$, together with at most 3 vertices (namely $u$, $v$, or $w$). 

\medskip

Using Examples 1 and 2, we deduce the following.

\begin{observation}\label{obs:d2b}
Let $H$ be an ordered graph such that $g_H(n)=\omega(\log n)$. Then $H$ has
maximum degree 1.
\end{observation}

\begin{proof}  Assume for the sake of contradiction that $H$ has a vertex $v$ of degree at least two. By Observation~\ref{obs:d2a}, all the neighbors of $v$ are predecessors of $v$ or all the neighbors of $v$ are successors of $v$. By symmetry, assume that all the neighbors of $v$ are predecessors of $v$. Observe that for any integer $i\ge 1$, the pair $(G_i,P_i)$ (defined above) does not contain $H$ as a pattern, since each vertex of $G_i$ has at most one neighbor in the graph $G_i-E(P_i)$ that precedes it, while $v$ has at least two neighbors that precede it in $H$. Since the longest induced path in $G_i$ has order $O(\log |V(G_i)|)$, this contradicts the assumption that  $g_H(n)=\omega(\log n)$ and concludes the proof.
\end{proof}

Observation \ref{obs:d2b} says that if we want to obtain
better-than-logarithmic lower bounds on the size of an induced path,
we need to study patterns of maximum degree at most 1, i.e.,
matchings. This is the topic of the next section.

\subsection{Matchings}\label{sec:matching}

In this section we investigate patterns consisting of ordered matchings, that is ordered
graphs in which all vertices have degree at most 1. We first prove a
result that will have several consequences of interest, the first one
being that we can restrict our study to ordered \emph{perfect
  matchings}, that is ordered 1-regular graphs.

\medskip

Let $G$ be an ordered graph with order $v_1,\ldots,v_n$, and let $H$
be an ordered subgraph of $G$ with vertex set $v_{a[1]}, v_{a[2]},\ldots,
v_{a[k]}$ (for $1\le a[1] < a[2] < \ldots < a[k] \le n$). The \emph{gap}
of $H$ in $G$ is defined as the minimum of $a[{i+1}]-a[i]$, for $1\le i
\le k-1$. The definition of the gap naturally extends to patterns in
pairs $(G,P)$ where $G$ is a graph and $P$ a Hamiltonian path in
$G$.

\begin{lemma}\label{lem:gap}
Let $G$ be a graph with a Hamiltonian path $P=v_1,\ldots,v_n$ ($n\ge 2$), and
with no induced path of order $t$. Let $H$
be a non-empty ordered matching. Then for any integer $m\le n/t$ such that
$g_H(\lfloor \tfrac m2 \rfloor)\ge t$, $(G,P)$ contains $H$ as a pattern with gap
at least $n/mt$.
\end{lemma}

\begin{proof}
We divide $P$ into $m$ vertex-disjoint subpaths $P_1,\ldots,P_m$, each of order
 at least $\lfloor n/m\rfloor$. For each path $P_i$, let $v_{a[i]}$ and $v_{b[i]}$ be the smallest
and largest vertices of $P_i$, respectively. Note that $b[i]+1=a[i+1]$ for any $1\le i \le m-1$. For each even $i$, we consider a shortest increasing path $Q_i$ in
$G$ between
$v_{b[i-1]}$ and $v_{a[i+1]}$, whose internal vertices are all in
$P_i$. Note that $Q_i$ is an induced path in $G$, and thus contains
fewer than $t$ vertices. It follows that  $Q_i$ contains two
consecutive vertices $v_{\alpha[i]}$ and $v_{\beta[i]}$ such that
$\beta[i]-\alpha[i]\ge (\lfloor n/m\rfloor +1)/t\ge n/mt\ge 1$. We divide $Q_i$ into two vertex-disjoint subpaths
$Q_i^-$ (with ends $v_{b[i-1]}$ and $v_{\alpha[i]}$) and $Q_i^+$ (with
ends $v_{\beta[i]}$ and $v_{a[i+1]}$).
We define $P_1'$ as the concatenation of $P_1$ and $Q_2^-$, and for each odd $3\le i\le m-1$ we define $P_i'$ as
the concatenation of $Q_{i-1}^+$, $P_i$, and $Q_{i+1}^-$. Note that the
concatenation of all paths $P_i'$ ($1\le i\le m-1$ odd) is
a path in $G$, which is increasing with respect to $P$.

We now define
a graph $G'$ with vertex set  $v_1',v_3',\ldots,v_{2s - 1}'$ with $s = \lfloor m/2 \rfloor$, with
an edge between $v_i'$ and $v_j'$ in $G'$ if and only there is an edge
between $P_i'$ and $P_j'$ in $G$. By definition,
$P':=v_1',v_3',\ldots,v_{2s - 1}'$ is a Hamiltonian path in $G'$ on
$s \ge $ vertices. Assume for the sake of contradiction
that $(G',P')$ avoids the pattern $H$. Then
$G'$ contains an induced path $Q'$ of order at least
$g_H(s)$. By definition of $G'$, it follows that $G$ also
contains an induced path $Q$ of order at least $g_H(s)\ge
t$. This induced path $Q$ in $G$ is obtained from $Q'$ as follows. Let
$Q'=v'_{c[1]}, v'_{c[2]}, \ldots, v'_{c[q]}$. Then we define $Q$ as a path in
$G$ of minimum order with the property that $Q$ is a concatenation of $q$ subpaths
$Q_1,\ldots,Q_q$ where for each $1\le i \le q$, all the vertices of
$Q_i$ lie on $P_{c[i]}'$. The existence of $Q'$ implies the existence of
such a path $Q$ (as each $P_{c[i]}'$ is a Hamiltonian graph in
$G$), and the minimality of $Q$ together with the property that $Q'$
is an induced path in $G'$ imply that $Q$ is an induced path in
$G$. We have obtained an induced path $Q$ of order $q \ge g_H(s)\ge
t$ in $G$, which is a contradiction. Hence, $(G',P')$ contains
the pattern $H$. As $H$ is a matching, $(G,P)$ also contains the
pattern $H$. Moreover, it follows from the definition of $G'$ that the
gap of the
corresponding 
pattern $H$ in $G$ is at least $n/mt$.
\end{proof}

Note that if $(G,P)$ contains a pattern $H$ with gap at least
$h+1$, then it contains the pattern $H^{+h}$ obtained from $H$ by adding $h$
isolated vertices between all pairs of consecutive vertices of $H$ (note that $H^{+h}$ contains $h(|V(H)|-1)+|V(H)|$ vertices and $|E(H)|$ edges). We
obtain the following simple corollary of Lemma \ref{lem:gap}.

\begin{corollary}\label{cor:deg0}
  Let $H$ be an ordered matching and $h\ge 1$ be a fixed  integer.
  \begin{itemize}
    \item If
$g_H(n)=\Omega(n^c)$ for some real number $0<c \le 1$, then
$g_{H^{+h}}(n)=\Omega(n^{c/(1+c)})$. \item If  $g_H(n)=\Omega((\log n)^c)$ for some real number $c>0$, then
$g_{H^{+h}}(n)=\Omega((\log n)^c)$.
\end{itemize}
\end{corollary}

\begin{proof}
Assume first that $g_H(n)\ge \alpha \cdot n^c$ for some real numbers
$\alpha>0$ and  $0<c \le
1$, and any sufficiently large $n$. Let $\beta:=\tfrac{\alpha^{1/(1+c)}}{(3h)^{c/(1+c)}}$, and
$t:=\beta\cdot n^{c/(1+c)}$. We will prove that $g_{H^{+h}}(n)\ge
t$ for any sufficiently large $n$. Let $G$ be a graph with a
Hamiltonian path $P=v_1,\ldots,v_n$ (with $n$ sufficiently large), and
assume
that $(G,P)$ avoids the pattern $H^{+h}$. Let
$m:=2(\beta/\alpha)^{1/c}\cdot n^{1/(1+c)}$. Then
\[g_H(\lfloor \tfrac{m}2\rfloor)\ge \alpha \cdot ((\beta/\alpha)^{1/c}\cdot
n^{1/(1+c)}))^c = \alpha \cdot (\beta/\alpha)\cdot
n^{c/(1+c)}= t.\]
Assume for the sake of contradiction that $G$ has
no induced path of order $t$. By Lemma \ref{lem:gap}, $(G,P)$
contains  $H$ as a pattern with gap at least \[\frac{n}{mt}\ge
  \frac{n}{(2(\beta/\alpha)^{1/c}\cdot n^{1/(1+c)}+1)\cdot \beta\cdot
    n^{c/(1+c)} }\ge
  \frac{\alpha^{1/c}}{2\beta^{1+1/c}+\alpha^{1/c}\cdot \beta \cdot n^{-1/(c+1)}}\ge
  \frac{\alpha^{1/c}}{3\beta^{1+1/c}},\]

  where the final inequality
holds for sufficiently large $n$. As
$\tfrac{\alpha^{1/c}}{3\beta^{1+1/c}}=h$, we have proved that $(G,P)$
contains the pattern $H^{+h}$, a contradiction.

\medskip

Assume now that $g_H(n)\ge \alpha \cdot (\log n)^c$ for some real numbers
$\alpha>0$ and  $c>0$, and any sufficiently large $n$. Let
$t:= \tfrac{\alpha}2\cdot (\log n)^{c}$. We will prove that $g_{H^{+h}}(n)\ge
t$ for any sufficiently large $n$. Let $G$ be a graph with a
Hamiltonian path $P=v_1,\ldots,v_n$ (with $n$ sufficiently large), and
assume
that $(G,P)$ avoids the pattern $H^{+h}$. Let $m:=2\exp(2^{-1/c}\log
n)$. Then \[g_H\left( \lfloor\tfrac m 2 \rfloor\right)\ge \alpha \cdot
  (2^{-1/c}\log n)^{c}  = t.\] Assume for the sake of contradiction that $G$ has
no induced path of order $t$. By Lemma \ref{lem:gap}, $(G,P)$
contains  $H$ as a pattern with gap at least \[\frac{n}{mt}\ge
  \frac{n}{(2\exp(2^{-1/c}\log
n)+1)\cdot \alpha \cdot
  (2^{-1/c}\log n)^{c}  }.\] The denominator is $o(n)$ for any $c>0$
and thus $n/mt$ is larger than $h$ for any sufficiently large
$n$. Thus $(G,P)$
contains the pattern $H^{+h}$, a contradiction.
\end{proof}

Observe that for any ordered matching $M$, there is an ordered perfect matching $H$
and an integer $h$ such that $M$ is an induced ordered subgraph of
$H^{+h}$ (it suffices to define $h$ as the maximum number of isolated
vertices between two vertices of degree 1 in $M$). We obtain the
following as a direct consequence.

\begin{corollary}\label{cor:deg0b}
Let $H$ be an ordered matching and let $H^-$ be the ordered perfect
matching obtained by deleting all vertices of degree 0 in $H$.
\begin{itemize}
  \item If
    $g_{H^-}(n)=\Omega(n^c)$ for
    some $0<c\le 1$, then $g_H(n)=\Omega(n^{c/(1+c)})$.
  \item If  $g_{H^-}(n)=\Omega((\log n)^c)$ for some $c>0$, then
$g_{H}(n)=\Omega((\log n)^c)$.
  \end{itemize}
\end{corollary}

Let $H$ be an ordered graph with ordering $u_1,\ldots,u_k$. We say
that two edges $e,f$ of $H$ are \emph{crossing} if there exist $1\le
a<b<c<d\le k$ such that $e=u_au_c$ and $f=u_bu_d$. If $H $ has no pair
of crossing edges then we say that $H$ is
\emph{non-crossing}.

\medskip

Using Corollary \ref{cor:deg0b}, we can restrict ourselves in the
remainder of this
section to ordered \emph{perfect} matchings.
Non-crossing ordered perfect matchings (i.e., 1-regular graphs) are in bijection with Dyck
words, in particular every non-empty non-crossing ordered perfect matching $C$ on vertices $u_1 \prec \ldots \prec u_k$ can be  uniquely
decomposed into two (possibly empty) non-crossing ordered perfect matchings
$A$ and $B$ such that $C$ contains an edge $u_1u_i$ (for some $i$),
$A$ is the ordered subgraph of $C$ induced by $u_2,\ldots,u_{i-1}$,
and $B$ is the ordered subgraph of $C$ induced by $u_{i+1},\ldots,
u_k$. We define the \emph{depth} of a non-crossing ordered perfect matching
$C$ inductively as follows: we set $\mathrm{depth}(C):=0$ if  $C$ is empty, and otherwise we decompose
$C$ into two non-crossing ordered perfect matchings $A$ and $B$ as above and we
set $\mathrm{depth}(C):=\max\{\mathrm{depth}(A)+1,\mathrm{depth}(B)\}$.

\medskip

We show that for ordered perfect matchings $H$, the growth rate of $g_H$ depends on
whether it has a pair of crossing edges or not. For non-crossing
perfect matchings, the growth is polynomial and the exponent depends only on the
depth of the perfect matching.

\begin{theorem}\label{tm:matchingBound}
Let $H$ be a non-empty non-crossing ordered perfect matching. Then $g_H(n)=\Omega(n^{1/d})$, where $d=\mathrm{depth}(H)$.
\end{theorem}
\begin{proof}
We prove the result by induction on $|E(H)|$. Let $d=\mathrm{depth}(H)$. If $|E(H)|=1$ then $d = 1$ and  $g_H(n) = n$ (since for any pair $(G,P)$ avoiding $H$, $E(G)=E(P)$). So we can assume that $|E(H)|\ge 2$.

Let $(A, B)$ be the unique decomposition of $H$ into two ordered perfect matchings described above. If $B$ is non-empty then $H$
is the concatenation of two non-empty ordered perfect matchings, each of
depth at most $d$, and each with fewer edges
than $H$. By the induction hypothesis and Lemma \ref{lem:concat}, we
obtain $g_H(n)=\Omega(n^{1/d})$, as desired.

So we can assume that $B$ is empty. Note that in this case $\mathrm{depth}(A)=d-1$. Let $G$ be a graph with a Hamiltonian path $P=v_1,\ldots,v_n$, and let $v_iv_j$ be an edge of $G$ with $i<j$ such that $j-i$ is maximal. If $j-i \le n^{1 - 1/ d}$, then by \Cref{obs:order}, there is an increasing induced path of order $\Omega(n^{1/ d})$ in $G$.
Otherwise, consider the subgraph $G'$ of $G$ induced by $\{v_{i+1}, \ldots,v_{j-1}\}$. The graph $G'$ avoids $A$, since if $A$ appears in $G'$ then together with the edge $v_iv_j$, one finds $H$ as a pattern in $(G,P)$. Since $|E(A)|<|E(H)|$ it follows from the induction hypothesis that  $G'$ (and thus $G$) contains an induced path of order $\Omega ( (n^{1 - \frac 1 d})^{\frac 1 {d-1}} ) = \Omega(n^{1 /{d}})$.
\end{proof}

Using Corollary \ref{cor:deg0b}, we obtain the following as a direct
consequence.

\begin{theorem}\label{thm:ncom}
Let $H$ be a non-empty non-crossing ordered matching, and let $H^-$ be
the ordered perfect matching obtained from $H$ by removing all isolated
vertices. Then
$g_H(n)=\Omega(n^{1/(d + 1)})$, where $d=\mathrm{depth}(H^-)$.
\end{theorem}

We have just proved that excluding patterns consisting of non-crossing ordered matchings provides
polynomial bound on the order of an induced path. We now show that
excluding (possibly crossing) ordered matchings as patterns  gives polylogarithmic bounds. 

\medskip

We start with the base case, where the forbidden pattern $M=$ \includegraphics[scale=1]{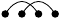} consists of four vertices and two crossing edges forming a matching. It was proved in \cite{esperet2017long} (using a different terminology) that in this case $g_M(n)\ge \tfrac12\log n$. For the sake of completeness, we give a significantly shorter and simpler proof.

\begin{lemma}\label{lem:1cross}
If $M=$ \includegraphics[scale=1]{patx.pdf}, then $g_M(n)\ge \tfrac12 \log n$.
\end{lemma}

\begin{proof}
We will prove the following stronger statement by induction on $n\ge 2$.
Let $G$ be a graph with a Hamiltonian path $v_1,\ldots,v_n$ such that $G$ (considered as an ordered graph with order $v_1,\ldots,v_n$) has no pair of crossing edges. Then $G$ admits two induced paths $L$ and $R$ which are increasing with respect to $v_1,\ldots,v_n$ and such that
\begin{itemize}
    \item $v_1$ is the starting vertex of $L$ and $v_n$ is the final vertex of $R$,
    \item $L < R$ (i.e., all vertices  $v_i\in L$ and $v_j\in R$ satisfy $i<j$), and
    \item $|V(L)| + |V(R)| \ge \log n$.
\end{itemize}

We can assume that $n\ge 3$, since otherwise the result follows by taking $L=\{v_1\}$ and  $R=\{v_n\}$.
We observe next that we can assume without loss of generality that  $v_1v_n$ is not an edge of $G$, since otherwise we can apply the result to $G-\{v_1v_n\}$ instead (an increasing induced path in $G-\{v_1v_n\}$ that contains $v_1$ and avoids $v_n$, or contains $v_n$ and avoids $v_1$ is also an increasing induced path in $G$). 

Consider now the set $F$ of \emph{maximal} edges of $G$ (i.e., edges $v_iv_j$ with $i<j$ such that there is no pair $(k,\ell)$ distinct from $(i,j)$ such that $k\le i<j\le \ell$ and $v_kv_\ell$ is an edge of $G$). Since $v_1v_n$ is not an edge of $G$, $t:=|F|\ge 2$. By maximality of the edges of $F$, and since there are no pairs of crossing edges, the edges of $F$ form an induced path which is increasing with respect to the order $v_1,\ldots,v_n$. Moreover, $F$ starts at $v_1$ and ends at $v_n$. Note that some edge $v_iv_j\in F$ is such that $j-i+1\ge n/t$. By induction, there are two increasing induced paths $L'$ and $R'$ in $G[v_i,\ldots v_j]$ such that $v_i$ is the starting vertex of $L'$ and $v_j$ is the final vertex of $R'$,  $L'<R'$, and $|V(L')| + |V(R')| \ge \log(n/t)$. Moreover, since $G$ has no pair of crossing edges and $v_iv_j$ is an edge of $G$, there is no edge between a vertex of $v_{i+1},\ldots,v_{j-1}$ and a vertex outside of $v_i,\ldots, v_j$. 

We define $L$ as the concatenation of the subpath of $F$ consisting of the edges between $v_1$ and $v_i$, together with the path $L'$; and $R$ as the concatenation of the subpath of $F$ consisting of the edges between $v_j$ and $v_n$, together with the path $R'$.

Note that $L$ and $R$ are both increasing induced paths such that $v_1$ is the starting vertex of $L$ and $v_n$ is the final vertex of $R$, and $L < R$. Furthermore,
$$ |V(L)| + |V(R)| \ge (t + 1) - 2 + \log(n/t) = t - 1 + \log n - \log t \ge \log n,
$$
since $t\ge 2$. This concludes the proof of the stronger statement.

\smallskip

It follows that $L$ or $R$ is an induced path of size at least $ \frac12 \log n$ in $G$, as desired.
\end{proof}

The logarithmic bound in Lemma~\ref{lem:1cross} is optimal up to a constant multiplicative factor, as shown by our motivating Example 2 (described after Observation \ref{obs:d2a}). This also follows from a result of Arocha and Valencia \cite{arocha2000long} that we state below.
For every integer $k\geq 1$ let $S_k$ 
be obtained from the $n$ points $v_1,\ldots,v_n$ of a regular $n$-gon with $n=3\cdot 2^k$ by adding straight-line edges between $v_i$ and $v_{i+2^j}$ for any $i\le 1 $ and $0\le j\le i$ (where indices are considered modulo $n$). Clearly $S_k$ is outerplanar and Hamiltonian. A depiction of $S_1$, $S_2$, and $S_3$ can be found on \Cref{fig:av}.

\begin{figure}[htb]
\centering
\begin{tikzpicture}[every node/.style = {draw=black,fill=black,circle, inner sep=0, minimum size=5pt}, every path/.style={line width=1pt}]
\begin{scope}[rotate = 90]
    \draw[color=teal] (0:1) node {} --
    (120:1)  node {} --
    (-120:1)  node {} -- cycle;
\end{scope}
\begin{scope}[xshift = 4cm, rotate = 90]
    \draw[color=teal] (0:1) node (a0) {} --
    (120:1)  node (a1) {} --
    (-120:1)  node (a2) {} -- cycle;
    \draw[color=magenta] (a0) -- (60:1) node (b0) {} --
    (a1) -- (180:1) node (b2) {} --
    (a2) -- (-60:1) node (b2) {} -- (a0);
\end{scope}
\begin{scope}[xshift = 8cm, rotate = 90]
    \draw[color=teal] (0:1) node (a0) {} --
    (120:1)  node (a1) {} --
    (-120:1)  node (a2) {} -- cycle;
    \draw[color=magenta] (a0) -- (60:1) node (b0) {} --
    (a1) -- (180:1) node (b1) {} --
    (a2) -- (-60:1) node (b2) {} -- (a0);
    \draw[color=orange] (a0) -- (30:1.5) node {} --
    (b0) -- (90:1.5) node {} --
    (a1) -- (150:1.5) node {} --
    (b1) -- (210:1.5) node {} --
    (a2) -- (270:1.5) node {} --
    (b2) -- (330:1.5) node {} --
    (a0);
\end{scope}

\end{tikzpicture}
\caption{First terms of a sequence of outerplanar graphs where induced paths have length at most logarithmic in the order of the graph.}
\label{fig:av}
\end{figure}

\begin{theorem}[\cite{arocha2000long}]\label{th:arocha}
There is a constant $c$ such that for every $k\in \N$, $S_k$ does not have an induced path of order more than $c\log |S_k|$.
\end{theorem}

In a graph $S_k$, we can choose a Hamiltonian path $P$ along the outerface and consider the ordered graph $(S_k,P)$. Since $S_k$ is outerplanar and $P$ follows the outerface, $(S_k,P)$ avoids the pattern \includegraphics[scale=1]{patx.pdf} (as well as any pattern that contains it). Since the $S_k$'s have an arbitrarily large number of vertices we have the following consequence.
\begin{corollary}\label{cor:1crossUB}
If $M=$ \includegraphics[scale=1]{patx.pdf}, then $g_M(n) = \Theta(\log n)$.
\end{corollary}

We now extend Lemma~\ref{lem:1cross} to all ordered matchings.

\begin{theorem}\label{thm:matchingplog}
For any  ordered matching $H$ with $|E(H)|\ge 2$, $g_H(n)=\Omega((\log n)^{1/d})$,
where $d=|E(H)|-1$.
\end{theorem}

\begin{proof}
Using Corollary \ref{cor:deg0b}, it is sufficient to prove the result in the case where  $H$ is a perfect matching.

We prove the result by induction on the number of edges of $H$. Let
$G$ be a graph with a Hamiltonian path $P=v_1,\ldots,v_n$, and assume
that $(G,P)$ does not contain the pattern $H$.  If $H$
has two edges then the result follows from Theorem \ref{thm:ncom} (if the edges of $H$ are non-crossing), and Lemma \ref{lem:1cross} (if the edges of $H$ are crossing).
So we can assume that $H$
contains at least three edges. Let $u_1,\ldots,u_k$ be the order on
the vertices of $H$, and let $u_j$ ($2\le j \le k$) be the unique neighbor of
$u_1$ in $H$. We denote the edge $u_1u_j$ by $e$, and define the
ordered graph $H^- $ as the ordered subgraph of $H$ obtained by
removing $u_1$ and $u_j$. Note that $H^-$ is a non-empty ordered
perfect matching, with $|E(H^-)|=|E(H)|-1$ and thus $g_{H^-}(n)=\Omega((\log
n)^{1/(d-1)})$ by the induction hypothesis.

Let $t= \epsilon \cdot (\log
n)^{1/d}$ for some sufficiently small constant $\epsilon>0$. If we find an induced path of order $t$ in $G$ then we are done, so we can assume in the remainder that $G$ does not contain an
induced path of order $t$.

Let $m=\exp ((\log
n)^{(d-1)/d})$. Then
\[
g_{H^-}(\lfloor \tfrac{m}2\rfloor)=\Omega\left ( \left ((\log
  n)^{(d-1)/d}\right )^{1/(d-1)}\right )=\Omega\left ((\log n)^{1/d}\right )\ge t.
\]
We can thus apply Lemma~\ref{lem:gap}, and obtain a copy $H^-_1$ of the pattern $H^-$ with
gap at least $n/mt$ in $G$. Fix such a copy $H^-_1$ in $G$, and recall that $H^-$ was obtained from $H$
by removing the (adjacent) vertices $u_1$ and $u_j$. Let $v_a$ and
$v_b$ be the images in $G$ of $u_{j-1}$ and $u_{j+1}$ from the copy
$H^-_1$ of $H^-$. Since $H^-_1$ has gap at least $n/mt$ in $G$, we have $b-a\ge n/mt$. We
now look at the subpath $P_1$ of $P$ with ends $v_a$ and $v_b$,
and consider the
subgraph $G_1$ of $G$ induced by the vertices of $P_1$. This graph has $n_1\ge
n/mt$ vertices. Note that $G$ has no edge with one endpoint before
the first vertex of the copy $H^-_1$, and one endpoint
in $P_1$, since otherwise $(G,P)$ would contain the pattern $H$.

We now apply Lemma \ref{lem:gap} again to  $(G_1,P_1)$ with the
same parameters $m$ and $t$, and find a copy $H_2^-$ of the pattern
$H^-$ with gap at least $n_1/mt\ge n/(mt)^2$ in $G_1$. Applying this
procedure $s:=\tfrac12(\log n)^{1/d}$ times, we find $s$ copies $H_1^-,H_2^-,
\ldots,H^-_s$ of the pattern $H^-$ in $G$ such that for any $1\le
i\le s-1$, the copy $H_{i+1}^-$ appears between the images of the
vertices $u_{j-1}$ and $u_{j+1}$ of the copy $H_{i}^-$ in $G$. Note
that there is a constant $c$ (depending on the constants in the big $O$) such that the final copy $H_s^-$ has gap at
least 
\begin{eqnarray*}
\frac{n}{(mt)^s}
&\ge&\frac{n}{\exp \left( \tfrac12(\log n)^{1/d}\cdot \left( (\log
      n)^{(d-1)/d}+ \log (\epsilon (\log n)^{\frac 1d}) \right) \right)}\\
&\ge&\frac{n}{\exp\left(\tfrac12(\log n)^{1/d}\cdot((\log
      n)^{(d-1)/d}+\tfrac1d \log \log n + c)\right)}\\
&=&\frac{n}{\exp\left(\tfrac12 (\log n) + \tfrac1{2d}(\log n)^{1/d}(\log \log n + dc)\right)}\\
&=& \frac{\sqrt{n}}{\exp(\tfrac1{2d}(\log n)^{1/d}(\log \log n + dc))}\ge 2.
\end{eqnarray*}
As the copy $H_s^-$ has gap at least 2 in $G$, there is a vertex $v_c$  appearing between the images of the
vertices $u_{j-1}$ and $u_{j+1}$ of the copy $H_{s}^-$ in $G$.
Consider a shortest increasing
path between $v_1$ and $v_c$ in $G$. No edge of this path can have its first
endpoint before the first vertex of a copy $H_i^-$, and the last endpoint between the images of the
vertices $u_{j-1}$ and $u_{j+1}$ of the copy $H_{i}^-$. It follows
that the shortest increasing path between $v_1$ and $v_c$ in $G$ has
size at least $s - 1=\Omega((\log n)^{1/d})$. This shortest path is in
particular an induced path in
$G$ of order $\Omega((\log n)^{1/d})$, as desired.
\end{proof}

\subsection*{Applications}

We now discuss several applications of Theorems \ref{tm:matchingBound} and  \ref{thm:matchingplog}. It turns out that these theorems allow us to quickly recover a number of known results, in a unified way, with much simpler proofs.

\medskip

We start with the following simple lemma, whose proof is very similar to that of Theorem~\ref{tm:matchingBound}. Let $H$ be an ordered matching on the (ordered) vertex set $u_2,\ldots,u_{k-1}$. We denote by $\widehat{H}$ the ordered matching obtained from $H$ by adding two vertices $u_1$ and $u_k$, joined by an edge (so the ordered vertex set of $\widehat{H}$ is $u_1,\ldots,u_k$,  and $|E(\widehat{H})|=|E(H)|+1$). See Figure \ref{fig:Hhat}, left, for an illustration.

\begin{lemma}\label{lem:hat}
For any ordered matching $H$, $g_{\widehat{H}}(n)\ge \min\{g_H(\sqrt{n}-1),\sqrt{n}\}$.
\end{lemma}

\begin{proof}
Let $G$ be a graph with a Hamiltonian path $P=v_1,\ldots,v_n$ and assume $(G, P)$ avoids the pattern $\widehat{H}$. Let $v_iv_j$ be an edge of $G$ with  $j - i$ maximal. If $j - i \le \sqrt{n}$ then by \Cref{obs:order} $G$ contains an induced path of order at least $\sqrt{n}$. Otherwise $j-i-1 \ge  \sqrt{n}-1$, and $v_{i+1} \dots v_{j-1}$ avoids $H$. Hence $P$ contains an induced path of order $g_H(\sqrt{n}-1)$.
\end{proof}

We immediately deduce the following, which was proved in \cite{esperet2017long}, with a significantly more complicated proof. We note that by using an asymmetric version of Lemma \ref{lem:hat}, the multiplicative constant $\tfrac14$ we obtain here can easily be optimized to match the multiplicative constant $\tfrac12-o(1)$ obtained in \cite{esperet2017long}.

\begin{theorem}\label{thm:k4m}
Any $K_{4}$-minor-free graph (and in particular any outerplanar graph) with a path of order $n$ contains an induced path of order $\tfrac14\log n$.
\end{theorem}

\begin{proof}
Recall that by Lemma \ref{lem:1cross}, if $M=$ \includegraphics[scale=1]{patx.pdf}, then $g_M(n)\ge \tfrac12 \log n$. By Lemma \ref{lem:hat}, it follows that $g_{\widehat{M}}(n)\ge \tfrac14 \log n$. It remains to observe that if $G$ contains a path $P$ of order $n$ in such a way that the pair $(G[V(P)],P)$ contains the pattern $\widehat{M}$, then $G$ contains $K_4$ as a minor.
\end{proof}

Similarly, we obtain the following result, which was proved in \cite{esperet2017long} in the specific case of planar graphs, with a significantly more complicated proof.

\begin{theorem}\label{thm:planar}
Any $K_{3,3}$-minor-free graph (and in particular any planar graph) with a path of order $n$ contains an induced path of order $\Omega(\sqrt{\log n})$.
\end{theorem}

\begin{proof}
Consider the ordered matching $H$ consisting of 6 vertices $u_2,\dots,u_7$, and edges $u_2u_5$, $u_3u_6$, and $u_4u_7$ (see Figure \ref{fig:Hhat}, left, for an illustration of $H$ and $\widehat{H}$). By Theorem \ref{thm:matchingplog}, $g_H(n)=\Omega(\sqrt{\log n})$, and by Lemma \ref{lem:hat}, $g_{\widehat{H}}(n)=\Omega(\sqrt{\log n})$. 

\begin{figure}[htb]
  \centering
    \includegraphics[scale=1.2]{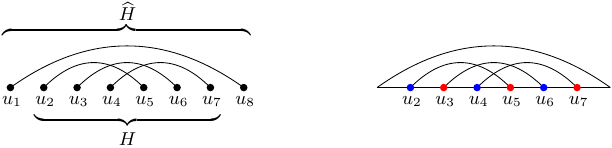}
  \caption{The ordered matchings $H$ and $\widehat{H}$ appearing in the proof of Theorem \ref{thm:planar}.\label{fig:Hhat}}
\end{figure}

Consider a $K_{3,3}$-minor-free graph $G$ with a path $P=v_1,\ldots,v_n$. By discarding the vertices of $G$ not in $P$, we can assume that $P$ is Hamiltonian. We claim that $G$ does not contain the pattern $\widehat{H}$, since otherwise $u_2,u_4,u_6$ on one side and $u_3,u_5,u_7$ on the other side would form a $K_{3,3}$-minor in $G$ (as illustrated in  Figure \ref{fig:Hhat}, right). The result then follows from the paragraph above.
\end{proof}

A similar argument yields the following more general result, which was also proved in \cite{esperet2017long} using a much more involved proof.

\begin{theorem}\label{thm:genus}
For any fixed surface $\Sigma$, any graph embeddable on $\Sigma$ which contains a path of order $n$ contains an induced path of order $\Omega(\sqrt{\log n})$.
\end{theorem}

\begin{proof}
Let $k$ be the Euler genus of $\Sigma$ ($k$ is a fixed constant), and let $H_k$ be the pattern consisting of the concatenation of $k+1$ consecutive copies of the ordered matching $\widehat{H}$ introduced in the proof of Theorem \ref{thm:planar} (and illustrated in Figure \ref{fig:Hhat}). Recall that by Lemma \ref{lem:hat}, $g_{\widehat{H}}(n)=\Omega(\sqrt{\log n})$, so it directly follows from Lemma \ref{lem:concat} that $g_{H_k}(n)=\Omega(\sqrt{\log n})$.
Note that the same argument as in the proof of Theorem \ref{thm:planar} shows that if a graph $G$ with a Hamiltonian path $v_1,\ldots,v_n$ contains the pattern $H_k$, then $G$ contains as a minor the disjoint union of $k+1$ copies of $K_{3,3}$. It directly follows from Euler's formula that the disjoint union of $k+1$ copies of $K_{3,3}$ is not embeddable on a surface of Euler genus $k$ (see Section 4.4 in \cite{MoharThom}), so we obtain a contradiction.
\end{proof}

\begin{figure}[htb]
  \centering
    \includegraphics[scale=1]{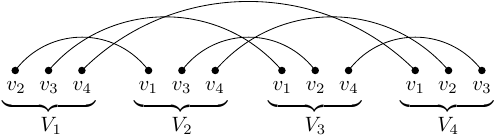}
  \caption{The ordered graph $\Pi(K_4)$ (where vertices are ordered from
    left to right).\label{fig:pi4}}
\end{figure}

Let $H$ be a graph. We now describe the construction of an
ordered matching $\Pi(H)$ such that containing the pattern $\Pi(H)$ implies containing $H$ as a minor.
We divide the vertex set $V(H)$ into $v_1, \dots, v_k$ and $u_1, \dots, u_p$ where the $u_i$'s are the isolated vertices of $H$.
For any $1\le i \le k$, we define $V_i$ as  the sequence of neighbors of $v_i$ in $H$ (in an arbitrary order).
The vertex set of $\Pi(H)$ is the concatenation of all sets $V_i$, for $1\le i \le k$, to which we concatenate $u_1, \dots ,u_p$.
Note that each vertex $v_i$ appears several times in $\Pi(H)$ (once in every set $V_j$ such  that $v_j$ is a neighbor of $v_i$ in $H$). For each pair of adjacent vertices $v_i,v_j$ in $H$, we add an edge in $\Pi(H)$ between the vertex $v_i$ of the set $V_j$ and the vertex $v_j$ of the set $V_i$.
Note that $\Pi(H)$ is a perfect matching on $\bigcup_{i \le k} V_i$, with $|E(H)|$ edges. See Figure \ref{fig:pi4} for a
picture of $\Pi(K_4)$). It can be observed that if a
graph $G$ with a Hamiltonian path $P$ is such that $(G,P)$ contains the pattern $\Pi(H)$, then
$G$ contains $H$ as a minor (it suffices to
contract each subpath of $P$ delimited by a set $V_i$ into a single
vertex). It follows that if a graph $G$ is $H$-minor-free, then for
any Hamiltonian path $P$ in $G$, the pair $(G,P)$ avoids the ordered
matching $\Pi(H)$ as a pattern. Since $\Pi(H)$ contains $|E(H)|$
edges, we obtain the following direct
consequence of Theorem \ref{thm:matchingplog}.

\begin{corollary}\label{cor:Hminorfree}
Let $H$ be a graph and let $G$ be an $H$-minor-free graph with a path of order $n$. Then
$G$ contains an induced path of order $\Omega((\log n)^{1/(|E(H)| - 1)})$. In particular, if $G$ is $K_h$-minor-free then
$G$ contains an induced path of order $(\log n)^{\Omega(1/h^2)}$.
\end{corollary}
A version of Corollary \ref{cor:Hminorfree} can also be obtained as a consequence of results in \cite{hilaire2023long}, but there the function of $h$ in the exponent was left unspecified (the proof there was based on the structure theorem for graphs excluding a topological minor by Grohe and Marx \cite{GM15}, itself based on the graph minor structure theorem of Robertson and Seymour, for which the currently best known explicit functions of $h$ are superexponential \cite{KTW20}). In a companion paper \cite{II} we will see how to improve the exponent in Corollary \ref{cor:Hminorfree} from $\Omega(1/h^2)$ to $\Omega(1/h(\log h)^2)$, even in the more general setting of forbidden topological minors. The arguments are significantly more involved.  

\medskip

We can now give a short proof of Corollary \ref{cor:minclosed} stated in the introduction, which says that when $\mathcal{C}$ is proper minor-closed, $f_\mathcal{C}(n)$ is polynomial if $\mathcal{C}$ exludes some outerplanar graph and polylogarithmic otherwise.

\begin{proof}[Proof of Corollary \ref{cor:minclosed}]
Assume first that $\mathcal{C}$ contains all outerplanar graphs. Then Theorem \ref{th:arocha} implies that $f_\mathcal{C}$ is at most logarithmic, while Corollary \ref{cor:Hminorfree} implies that $f_\mathcal{C}$ is at least polylogarithmic.

\begin{figure}[htb]
  \centering
    \includegraphics[scale=1]{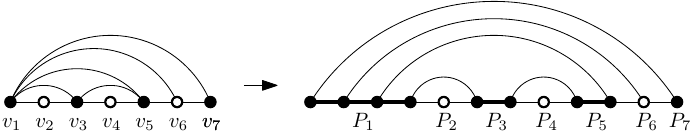}
  \caption{Excluding some outerplanar graph as a minor is equivalent to excluding a non-crossing matching as a pattern.\label{fig:outer}}
\end{figure}

Assume now that $\mathcal{C}$ excludes some outerplanar graph $H'$. In particular $\mathcal{C}$ excludes all supergraphs of $H'$, and thus some maximal outerplanar graph $H$. As $H$ is maximal outerplanar, it is Hamiltonian. Let $P=v_1,\ldots,v_k$ be a Hamiltonian path in $H$. For each $1\le i \le k$, replace $v_i$ by a path $P_i$ of vertices that are each incident to a distinct edge of $E(H)-E(P)$ incident to $v_i$ (see Figure \ref{fig:outer}). The resulting graph is an outerplanar graph consisting of the union of a Hamiltonian path and a non-crossing matching, and if a graph with a Hamiltonian path contains the corresponding  matching as a pattern, then it also contains $H$ as a minor (simply contract each subpath of the Hamiltonian path corresponding to some $P_i$ into a single vertex $v_i$). As all graphs from $\mathcal{C}$ exclude $H$ as a minor, it follows from Theorem \ref{thm:ncom} that $f_\mathcal{C}$ is polynomial. 
\end{proof}

Note that any graph $H$ from a family of bounded degree expanders has treewidth $\Theta(|V(H)|)=\Theta(|E(H)|)$. In particular, if a graph has treewidth at most $t$, it avoids some expander $H_t$ with $|E(H_T)| = O(t)$ as a minor. We thus obtain the following consequence of Corollary \ref{cor:Hminorfree}.

\begin{corollary}\label{cor:tw}
There exists a constant $c>0$ such that any graph $G$ with a path of order $n$ and with treewidth at most $t$ has an induced path of order $\Omega((\log n)^{c/t})$.
\end{corollary}

A similar result was obtained in \cite{hilaire2023long} with a different proof, and a better constant $c$ in the exponent.

\medskip

We conclude this section with a quick application of Theorem \ref{tm:matchingBound} to graphs of bounded pathwidth (the same result was proved in \cite{hilaire2023long} using a different technique). The \emph{pathwidth} of a graph $G$ is denoted by $\mathrm{pw}(G)$. A \emph{complete ternary tree} of height $t$ is a rooted tree in which every non-leaf node has three children
and every leaf is at distance $t$ from the root. Such a tree has $3^{t+1} - 1$ nodes. It is known that the pathwidth of a complete ternary tree of height $t$ is $t$ (see \cite{GJWW}), and that if $H$ is a minor of $G$ we have $\text{pw}(H) \le \text{pw}(G)$. 

\begin{corollary}\label{cor:pw_bound}
If a graph $G$ has pathwidth less than $t$ and contains a path on $n$ vertices, then $G$ contains an induced path on $\Omega(n^{1/t})$ vertices.
\end{corollary}
\begin{proof} 
We construct a sequence of ordered perfect matchings $(M_k)_{k \ge 1}$ such that for any graph $G$ with a Hamiltonian path $P$, either $(G,P)$ avoids the pattern $M_k$ or $\mathrm{pw}(G)\ge k$.

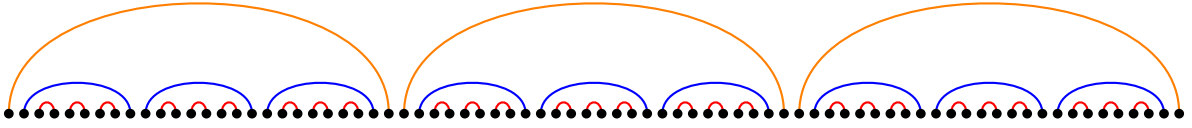
\begin{figure}[htb]
\centering
\begin{tikzpicture}[scale = 0.2, every node/.style = {black node, minimum size=0.1cm}]
\foreach \o in {1, 27, 53}{
    \draw[orange, thick] (\o-1, 0) node {} to[bend left=90] ++ (25,0) node {};
    \begin{scope}[xshift=\o cm]
        \foreach \s in {1,9,17}{
            \draw[blue, thick] (\s-1, 0) node {} to[bend left=90] ++ (7,0) node{};
            \foreach \i in {0,2,4}{
                \begin{scope}[xshift = \s cm]
                    \draw[red, thick] (\i, 0) node {} to[controls=+(90:1) and +(90:1)] ++(1,0) node {};
                \end{scope}
            }
        }
    \end{scope}
}
\end{tikzpicture}
\caption{The non-crossing ordered perfect matching $M_3$ used in the proof of \Cref{cor:pw_bound}.}
\label{fig:pw_pat}
\end{figure}

The sequence $(M_k)_{k \ge 0}$ is defined inductively as follows: $M_0$ is the empty graph, and for any $k\ge 1$, $M_k=\widehat{M_{k-1}}\cdot \widehat{M_{k-1}}\cdot \widehat{M_{k-1}}$ (we recall that $\cdot$ denotes the concatenation, and the operator $H\mapsto \widehat{H}$ was defined shortly before Lemma \ref{lem:hat}). See \Cref{fig:pw_pat} for an illustration of $M_3$.
Suppose a Hamiltonian graph $G$ contains the pattern $M_k$. Then one can find a ternary tree of height $k$ as a minor in $G$ by contracting together the three larger edges of $M_k$, then for each copy of $M_{k-1}$, contracting together the three larger edges of the copy of $M_{k-1}$ and so on.

As the pathwidth of a ternary tree of height $k$ is exactly $k$, Hamiltonian graphs of pathwidth less than $k$ (ordered  according to their Hamiltonian path) avoid the pattern $M_k$.
Since $\mathrm{depth}(M_k)=k$ for any $k\ge 0$, it follows from \Cref{tm:matchingBound} that if a Hamiltonian graph has pathwidth less than $t$, it contains an induced path on $\Omega(n^{1/t})$ vertices.
\end{proof}

\subsection{General patterns}
\label{sec:triple}

We start by giving the following immediate consequence of \Cref{th:main} (on the exclusion of $K_{t,t}$ as a subgraph) for the exclusion of an \emph{ordered} bipartite subgraph.
\begin{theorem}\label{th:ordbip}
Let $v_1 \dots v_h$ denote the vertices of an ordered bipartite graph $H$.
If for some $i\in \intv{1}{h}$, every edge of $H$ is of the form $v_jv_{j'}$ with $j\leq i<j'$ then $g_H(n) = \Omega_h((\log \log n/\log \log \log n)^{1/5})$.
\end{theorem}
\begin{proof}
Observe that every arbitrarily ordered $K_{2h + 2,2h + 2}$ contains as a pattern an ordered $K_{h,h}$ where all the vertices from one color class come before the other vertices. This ordered $K_{h,h}$ contains $H$ as a pattern. So if $G$ is a $n$-vertex graph  with a Hamiltonian path $P$ such that $(G,P)$ avoids the pattern $H$, then $G$ is $K_{2h + 2,2h + 2}$-subgraph free. So by \Cref{th:main} $G$ has an induced path on  $\Omega_h((\log \log n/\log \log \log n)^{1/5})$ vertices.
\end{proof}

Recall that a graph is said to be \emph{$d$-degenerate} (for some $d\in \N$) if every subgraph of it has a vertex of degree at most $d$. We have the following consequence of \Cref{th:main}.

\begin{corollary}\label{cor:bddeg}
If $\mathcal{C}$ is the class of $d$-degenerate graphs, then \[f_{\mathcal{C}}(n) \geq \Omega_d((\log \log n/\log \log \log n)^{1/5})).\]
\end{corollary}
Indeed, $d$-degenerate graphs exclude $K_{d+1,d+1}$ as subgraphs, and we can apply \Cref{th:main}. This bound however is not optimal: Ne{\v{s}}et{\v{r}}il and Ossona de Mendez proved in \cite{nevsetvril2012sparsity} that $f_{\mathcal{C}}(n) = \Omega_d(\log \log n)$. As we will recall in \Cref{sec:open}, it remains an open problem whether the general bound of \Cref{th:ordbip} can be improved to match the aforementioned $\Omega(\log \log n)$.

\medskip

In the remainder of this section, we give lower a bound on $g_H$ when $H$ is an ordered half-graph (\Cref{th:log3}). This is a corollary of a translation (\Cref{thm:GRS})  of Galvin, Rival, and Sands' original proof \cite{galvin1982ramsey} of \Cref{th:galvin} to the setting of graphs avoiding patterns.

\medskip

For $n\geqslant 1$, the \emph{ordered half-graph} $\mathcal{H}_{n}$ is the ordered graph with ordered vertex set $a_1,b_1,\ldots,a_n,b_n$ and an edge between $a_i$ and $b_j$ if and only if $i<j$.

\begin{theorem}[from \cite{galvin1982ramsey}]\label{thm:GRS}
Let $G$ be a graph with a Hamiltonian path $v_1,\ldots,v_n$, with $n\ge 2$. Then there is $p=\Theta((\log^{(3)} n/\log^{(4)}n)^{1/3})$ such that $G$ (considered as an ordered graph with order $v_1,\ldots,v_n$) contains an ordered half-graph $\mathcal{H}_{p/4}$ as a pattern, or $G$ contains an increasing induced path with $p$ vertices.
\end{theorem}

\begin{proof} We define $p$ as the smallest positive  integer which is a multiple of 4 and is such that $p^3\ge \tfrac{\log^{(3)} n}{6\log^{(4)} n}$. For $n$ large enough we have $p^3\le \tfrac{\log^{(3)} n}{5\log^{(4)} n}$ and thus $\log p \le \log^{(4)} n-\log^{(5)}n-\log 5 \le \log^{(4)} n$. It follows that $\log^{(4)}n\ge \log p$. Since $\log^{(3)} n \ge 5p^3\log^{(4)}n$, 
we have
\[\log^{(3)}n\ge 5 p^3 \log p.\]
As a consequence, for $n$ (and thus $p$) large enough,
\[\log^{(2)}n\ge p^{5p^3}\ge  p^{4p^2(p-3)+4}\log(p^6)+\log^{(2)}p^2.\]
This implies
\[\log n\ge\left((p^6) \pow (p^4) \pow (p^2(p-3)+1)\right)\cdot \log p^2,\]
and thus
\[n\ge (p^2) \pow (p^6) \pow (p^4) \pow (p^2(p-3)+1)\ge R_{p^2}(p+1;4).\]

\smallskip

For every pair $(i,j)$ with $1\leq i<j\leq n$, let us fix an increasing induced path $P_{i,j}$ from $v_i$ to $v_j$ (such a path always exist, for instance consider any shortest increasing path from $v_i$ to $v_j$ in $G[\intv{v_i}{v_j}]$).
If for some such pair, $P_{i,j}$ has at least $p$ vertices then we are done, so in the following we assume that all such paths have at most $p-1$ vertices

We consider the 4-clique on the vertex set $V(G)$.
To every hyperedge $\{i,j,k,\ell\}$ (with $1\leq i<j<k<\ell\leq n$) we assign a color as follows:
\begin{itemize}
    \item If there is no edge in $G$ between $P_{i,j}$ and $P_{k,\ell}$, we assign color 0 to $\{i,j,k,\ell\}$.
    \item Otherwise we arbitrarily choose an edge $e$ between these two paths. Let $s$ and $t$ respectively denote the index of the endpoints of $e$ along $P_{i,j}$ and $P_{k,\ell}$ (as these paths have at most $p-1$ vertices, each index is an element of $[p-1]$). We give color $(s,t)$ to $\{i,j,k,\ell\}$
\end{itemize}
As noted above the coloring we defined uses at most $(p-1)^2+1\leq p^2$ colors.
Applying Erd\H{o}s and Rado's bound on the multicolored Ramsey number (\Cref{th:ramseyc}) for $k = 4$, $q = p^2$ and $N=p + 1$ and the above inequality on $n$, the above hypergraph contains a 4-clique on $2p$ vertices with a unique color $c$. Let $v_{\pi(1)}, \dots, v_{\pi(N)}$ denote these vertices, with $\pi(1)<\dots<\pi(N)$. 

Suppose first that $c=0$. Consider a shortest increasing path $Q$ from $v_{\pi(1)}$ to $v_{\pi(N)}$ in the ordered subgraph of $G$ induced by the vertices of the paths $P_{\pi(1),\pi(2)},\dots, P_{\pi(N-1), \pi(N)}$.
As $c=0$, for every $i,j\in \intv{1}{N - 1}$ with $i+1<j$ there is no edge from $P_{\pi(i),\pi(i+1)}$ to $P_{\pi(j), \pi(j+1)}$.
In particular for every $i\in \intv{2}{N-1}$ there is no edge from any $P_{\pi(j), \pi(j+1)}$ with $j+1\leq i$ to any $P_{\pi(j'), \pi(j'+1)}$ with $j'\geq i+1$. It follows that $Q$ contains at least one vertex of each path $P_{\pi(j), \pi(j+1)}-\{v_{\pi(j+1)}\}$, and thus $Q$ contains at least $N - 1 = p$ vertices. As $Q$ is an increasing induced path in $G$, we obtain the desired result.

Assume now that $c=(s,t)$, for some indices $s,t\in [p-1]$. For any $0\le i \le p/4-1$, we let $a_i$ be the vertex of index $s$ on $P_{\pi(4i+1),\pi(4i+2)}$, and $b_i$ be the vertex of index $t$ on $P_{\pi(4i+3),\pi(4i+4)}$. Then the sequence $a_1,b_1,\ldots,a_{p/4}, b_{p/4}$ witnesses that $G$ contains the ordered half-graph $\mathcal{H}_{p/4}$ as a pattern.
\end{proof}

Recall Observation \ref{obs:d2a}, which says that ordered graphs $H$ such that $\{g_H(n): n\in \mathbb{N}\}$ is unbounded must be bipartite, and moreover for each vertex $v\in V(H)$, all neighbors of $v$ are 
predecessors of  $v$, or all neighbors of $v$ are  successors of $v$. 

Using Theorem \ref{thm:GRS}, we now prove that for any graph $H$ as above, the growth of $g_H(n)$ is uniformly bounded by a triple logarithm. Moreover, we can always find long \emph{increasing} induced paths when the pattern $H$ is avoided. 

\begin{theorem}\label{th:log3}
Let $H$ be such that $\{g_H(n): n\in \mathbb{N}\}$ is unbounded. Then any graph $G$ with a Hamiltonian path $v_1,\ldots,v_n$ and which avoids the pattern $H$ contains an induced path of $\Omega((\log^{(3)}n/\log^{(4)}n)^{1/3})$ vertices which is increasing with respect to $v_1,\ldots,v_n$. In particular $g_H(n)=\Omega((\log^{(3)}n/\log^{(4)}n)^{1/3})$.
\end{theorem}

\begin{proof}
By Observation~\ref{obs:d2a}, $H$ has a bipartition $U,V$ such that for any $u\in U$, all neighbors of $u$ are 
successors of  $u$, and for any $v\in V$, all neighbors of $v$ are  predecessors of $v$. Let $h=|V(H)|$. Consider a graph $G$ with a Hamiltonian path $v_1,\ldots,v_n$. By Theorem \ref{thm:GRS}, if $G$ has no increasing induced path of order $\Omega((\log^{(3)}n/\log^{(4)}n)^{1/3})$, it contains an ordered half-graph whose size (as a function of $n$) is unbounded. In particular for $n$ large enough, there exist integers \[1\le \pi(1)<\phi(1)<\cdots < \pi(4h)<\phi(4h)\le n\] such that for any $1\le i <j\le 4h$,  $v_{\pi(i)}$ is adjacent to $v_{\phi(j)}$. Denote the (ordered) vertex set of $H$ by $u_1,\ldots,u_h$. For every $1\le i \le h$, map $u_i$ to $v_{\pi(4i-3)}$ if $u_i\in U$, and otherwise map $u_i$ to $v_{\phi(4i-1)}$. Note that the subgraph of $G$ induced by the chosen vertices is a supergraph of $H$, moreover the images of any two vertices $u_i,u_j$ in $G$ are non-consecutive in the Hamiltonian path $v_1,\ldots,v_n$. It follows that the ordered graph $G$ with the Hamiltonian path $v_1,\ldots,v_n$ contains $H$ as a pattern.
\end{proof}

We conclude with the proof of the dichotomies of \Cref{cor:dicho}
\begin{proof}[Proof of \Cref{cor:dicho}]
The items of the statement are direct consequences of the following results:
\begin{enumerate}
    \item \Cref{obs:d2b}, \Cref{cor:1crossUB}, and \Cref{thm:ncom};
    \item \Cref{th:ordbip};
    \item \Cref{obs:d2a} and \Cref{th:log3}.
\end{enumerate}
\end{proof}

\section{Open problems}\label{sec:open}

Recall that the work described in this paper revolves around the two following questions:

\fstquestion*

\sndquestion*

Regarding \Cref{question}, we recall that the currently best bounds are summarized in \Cref{fig:survey}.
Three different behaviors of the function $f$ are highlighted there: polynomial, polylogarithmic and doubly polylogarithmic. As noted in the introduction, for hereditary graph classes the exclusion of a bipartite subgraph is the most general setting where the property of \Cref{question} holds with an unbounded function $f$, and we proved above that in this case, $f(n) = \Omega(\polyloglog(n))$. In
\Cref{cor:minclosed} we have obtained a characterization of the minor-closed classes where $f$ is polynomial (and showed that there is a gap between polynomial and polylogarithmic complexities).
However, there is currently no known characterization of the hereditary graph classes where $f$ is at least polylogarithmic or polynomial. We leave these as open problems.

A second goal is to close the gaps between known upper- and lower-bounds on the function~$f$, in particular in the following settings:
\medskip 
\begin{center}
\begin{tabular}{llll}
 class & lower-bound & upper-bound & references\\ 
 \hline 
 graphs of $\pw < k$ & $\frac{1}{3} n^{1/k}$ & $n^{2/k} + 1$ & \cite{hilaire2023long}\\[0.5em]  
 graphs of $\tw < k$ & $\frac{1}{4} (\log n)^{1/k}$ & $(k+1) (\log n)^{2/(k-1)}$& \cite{hilaire2023long, esperet2017long}\\[0.5em]
 planar graphs & $\left (\frac{1}{2\sqrt{6}} - o(1) \right ) \sqrt{\log n}$ & $\frac{3 \log n}{\log \log n}$& \cite{esperet2017long}\\[0.5em]
 $d$-degenerate & $\log\log n/\log (d+1)$ & $O((\log \log n)^2)$ & \cite{nevsetvril2012sparsity, defrain2024sparse}\\[0.5em]
 $K_{t,t}$-subgraph-free & $\Omega\left (\left (\frac{\log \log n}{\log \log \log n}\right)^{1/5}\right)$ &$O((\log \log n)^2)$ & this paper and \cite{defrain2024sparse}.
 \end{tabular}
\end{center}
\medskip

Finally, one could more generally investigate graph classes avoiding an arbitrary (finite or infinite) number of patterns, instead of a single pattern. Does the function $g$ in this    case fall within the few different behaviors observed above?

\subsection*{Acknowledgements}

The  authors would like to thank Stéphan Thomassé for stimulating discussions at the beginning of this project,  Clément Rambaud for suggesting Corollary~\ref{cor:minclosed}, and the reviewers for their helpful remarks and suggestions.

\subsection*{Postscript}
After this article was submitted, a tight lower-bound for planar graphs was obtained by Hugo Jacob and the first author in \cite{duron2025planar}, and the bound in Theorem \ref{th:main} was improved to $c \log \log n/\log \log \log n$ by Hunter, Milojevi\'c, Sudakov, and Tomon in \cite{HMST}.

\bibliographystyle{alpha}
\bibliography{references}

\end{document}